\documentclass[a4paper,oneside]{article}
\usepackage{amssymb,amsmath,amsthm,bm,booktabs,caption,graphicx}
\usepackage[font={sf,small}]{floatrow}
\newcommand{\norm}[1]{\left\lVert#1\right\rVert}
\newcommand{\vertiii}[1]{\left|#1\right|}
\newcommand{\bmnabla}{\bm\nabla}
\newcommand{\bint}[2]{\langle#1\rangle_{#2}}
\newcommand{\lam}{\lambda}
\newcommand{\bmsigma}{\bm\sigma}
\newcommand{\inpro}[1]{\langle#1\rangle_h}
\newcommand{\ninpro}[1]{\langle#1\rangle_{\mathcal A_h}}
\newcommand{\we}{\widetilde}
\newcommand{\bndint}[1]{\sum_{T\in\mathcal T_h}\langle#1\rangle_{\partial T}}
\newtheorem{lem}{Lemma}[section]
\newtheorem{thm}{Theorem}[section]
\newtheorem{pro}{Proposition}[section]
\newtheorem{rem}{Remark}[section]
\newtheorem{assumption}{Assumption}
\newtheorem{algorithm}{Algorithm}

\numberwithin{equation}{section}
\begin{document}
\title
{
  \Large\bf Analysis of a two-level algorithm for   HDG methods for diffusion problems \thanks
  {
    This work was supported in part by National Natural Science Foundation of
    China (11171239, 11401407) and Major Research
    Plan of  National Natural Science Foundation of China (91430105).
  }
}
\author
{
  Binjie Li\thanks{Email: libinjiefem@yahoo.com},
  Xiaoping Xie \thanks{Corresponding author. Email: xpxie@scu.edu.cn},
  Shiquan Zhang\thanks{Email:  shiquanzhang@scu.edu.cn}\\
  {School of Mathematics, Sichuan University, Chengdu 610064, China}
}
\date{}
\maketitle
\begin{abstract}
  This paper analyzes an abstract two-level algorithm for
  hybridizable discontinuous Galerkin (HDG) methods in a unified fashion.
  We use an extended version of the Xu-Zikatanov (X-Z)   identity to derive a sharp estimate of the convergence rate of the algorithm, and show that the theoretical results also apply to   weak Galerkin (WG) methods.
  The main features of our analysis are twofold: one is that  we only need the minimal   regularity of the model problem;
  the other is that we do not require the triangulations to be quasi-uniform.
  Numerical experiments are provided  to confirm the   theoretical results.
\end{abstract}
\section{Introduction}

The Hybridizable Discontinuous
Galerkin (HDG) framework,    proposed in
\cite{Cockburn;unified HDG;2009} (2009) for second order elliptic problems, provides a unifying strategy for
hybridization of finite element methods.   The  unifying framework includes as particular cases hybridized versions of mixed methods \cite{ArnoldBrezzi1985, BrezziDouglasMarini1985,COCKBURN-GOPALAKRISHNAN2004}, the continuous Galerkin (CG) method \cite{Cockburn-Gopalakrishnan-Wang2007}, and a   wide class
of hybridizable discontinuous Galerkin (HDG) methods.
Here hybridization denotes the process to rewrite a finite element method as a hybrid version. It should be pointed that
the Raviart-Thomas (RT) \cite{RT} and Brezzi-Douglas-Marini (BDM) mixed methods were first shown in \cite{ArnoldBrezzi1985, BrezziDouglasMarini1985} to have equivalent
hybridized versions, and an overview of some  hybridization techniques was presented
in \cite{COCKBURN_2005}.
In the  so-called HDG methods following the HDG framework,   the constraint of function continuity on the inter-element boundaries is relaxed by introducing  Lagrange multipliers defined on the
the inter-element boundaries, thus allowing for piecewise-independent approximation to the potential or flux solution. By local elimination of the unknowns defined in the interior of elements, the HDG methods finally
lead to symmetric and positive definite (SPD) systems where the unknowns are only the globally coupled degrees of freedom describing the Lagrange multipliers.
We refer to \cite{super_con_LDG, projection_based_hdg,Li;hdg;2014} for the convergence analysis of several HDG methods for the second order elliptic problems.

Closely related to the HDG framework is the    weak Galerkin (WG) finite element method
\cite{Wang;2013, Mu-Wang-Wang-Ye,Mu-Wang-Ye1, Mu-Wang-Ye2} pioneered by Wang and Ye \cite{Wang;2013}. The WG method is designed by using a weakly defined gradient operator over
functions with discontinuity, and then allows the use of totally
discontinuous piecewise polynomials in the finite element procedure. %,  Nguyen-Peraire-Cockburn
By introducing the discrete weak gradient as an independent variable, as shown in \cite{Li;mg_wg;2014}, the WG method can be rewritten as some
HDG version when the diffusion-dispersion tensor in the corresponding second order elliptic equation is a piecewise-constant matrix.

It  is well-known that the design of fast solvers  is a key component to numerically solving  partial differential equations. For the HDG methods as well as the WG methods, so far there are only limited literature concerning this issue.
In \cite{Gopalakrishnan;2009} (2009), Gopalakrishnan and Tang analyzed a $\cal V$-cycle multigrid
algorithm for two type of HDG methods for the Poisson problem with full elliptic
regularity. By following the same idea,
Cockburn et al. \cite{Cockburn;2014} (2014) presented the first convergence study of a nonnested
$\cal V$-cycle multigrid algorithm for one type of HDG method for diffusion equations without full elliptic regularity.
Chen  et al. \cite{Chen;2014} (2014) constructed two auxiliary space multigrid
preconditioners for two types of WG methods for the diffusion equations. In \cite{Li;mg_wg;2014} Li and Xie proposed a
two-level algorithm for two types of WG methods without full elliptic regularity, and,   in
\cite{Li;bpx;2014}, they analyzed an optimal BPX preconditioner for   a large class of  nonstandard finite element methods  for the diffusion equations,
including the   hybridized Raviart-Thomas  and Brezzi-Douglas-Marini   mixed element
methods, the hybridized discontinuous Galerkin  method, the Weak Galerkin  method, and the
nonconforming Crouzeix-Raviart   element method.

In this paper, we shall propose and analyze an abstract two-level algorithm for the  SPD systems arising from the HDG methods for the following diffusion model:
\begin{equation}\label{eq:model}
  \left\{
    \begin{array}{rcll}
      -\text{div}({\bm a}\bmnabla u) &= & f & \text{in $\Omega$},\\
      u &=& 0 & \text{on $\partial\Omega$},
    \end{array}
  \right.
\end{equation}
where $\Omega\subset R^d~(d=2,3)$ is assumed to be a bounded polyhedral domain, the diffusion-dispersion tensor ${\bm a}\in [L^{\infty}(\Omega)]^{d\times d}$ is a SPD
matrix and $f\in L^2(\Omega)$. In the  two-level algorithm,  the
$H^1$-conforming piecewise linear finite element space  is used as the auxiliary space.  The main tool of our analysis is an extended version of the Xu-Zikatanov (X-Z) identity \cite{Xu;2002}. % which is derived in \cite{Li;mg_wg;2014} by   following the basic ideas of \cite{Xu;2002, Xu;2009,Chen;2010}.
The main features of our work are as follows:
\begin{itemize}
\item We only need the minimal   regularity of the model problem
  \eqref{eq:model} in the sense that  the regularity estimate
  \begin{equation}\label{eq:reg}
    \norm{u}_{1+\alpha,\Omega}\leqslant C_{\Omega}\norm{f}_{\alpha-1,\Omega}
  \end{equation}
  holds with  $\alpha\in [0,1]$, where   $C_{\Omega}$ is a positive constant that only depends on $\Omega$ and
  $\bm a$.
  Based on the convergence results of the two-level algorithm, {\bf Algorithm
\ref{algo:two-level}}, (cf. Theorems \ref{thm:convergence}-\ref{thm:convergence2}),
  it is easy to show that the multigrid methods which fall into the proposed  two-level algorithm
  framework  for the HDG methods all converge.
  We note that the analyses in \cite{Gopalakrishnan;2009} and \cite{Cockburn;2014} require full regularity ($\alpha=1$)  and $\alpha\in (0.5,1]$, respectively.

\item We only assume the grids to be
  conforming and shape regular. Thus,   the quasi-uniform condition, which is assumed in \cite{Gopalakrishnan;2009,Cockburn;2014,Chen;2014,Li;mg_wg;2014,Li;bpx;2014}, is not required in our analysis. Therefore, based on  fast solvers for   the
  auxiliary space, our analysis can be used to design fast solvers on
  adaptively refined grids and completely unstructured grids.

\item Our theoretic results  also apply to the WG methods.
\end{itemize}

% .  However, it required that \eqref{eq:regularity} hold with $\alpha\in (0.5,1]$ and $\bm a$ be piecewise constant.
% \par
% All the fast solvers mentioned above, except the preconditioners and the two-level algorithm presented
% in \cite{Li;mg_wg;2014}, require model problem \eqref{eq:model} admit the regularity estimate
% where $\alpha\in (0.5,1]$ and $C_{\Omega}$ is a positive constant that only depends on $\Omega$ and
% $\bm a$. And they all require the grids to be quasi-uniform.

The rest of this paper is organized as follows. Section \ref{sec:notation} introduces
notations, an extended version of X-Z identity and   HDG methods.
Section \ref{sec:mg} describes and analyzes the two-level algorithm. Section \ref{sec:appl} presents
some applications of the algorithm to the HDG methods as well as to the WG methods. Section \ref{sec:numerical}
reports some numerical results to verify the theoretic results.
\section{Preliminaries}\label{sec:notation}
\subsection{Notations}\label{ssec:basic}
For an arbitrary open set $D\subset\mathbb R^d$, we denote by $H^1(D)$ the   Sobolev space of scalar
functions on $D$
 whose derivatives up to order $1$ are square integrable, with the norm $\norm{\cdot}_{1,D}$. The
notation $ |\cdot|_{1,D}$ denotes the semi-norm derived from the partial derivatives of order equal to $1$.
 The space $H_0^1(D)$ denotes the closure in $H^1(D)$ of the set of infinitely differentiable
functions with compact supports in $D$. 
We use $(\cdot,\cdot)_D$ and $\bint{\cdot,\cdot} {\partial D}$ to denote
 the $L^2$-inner products on the square integrable function spaces   $L^2(D)$ and $L^2(\partial D)$,  respectively, with $\norm{\cdot}_D$
and $\norm{\cdot}_{\partial D}$  representing the corresponding induced $L^2$-norms.  Let $P_k(D)$ denotes the set of
polynomials of degree $\leqslant k$ defined on $D$.
%For each
%$w\in H^1(D):=\{v\in L^2(D):\bmnabla v\in [L^2(D)]^d\}$ , we denote
%\begin{displaymath}
%  \begin{split}
%    |w|_{1,D} := \norm{\bmnabla w}_D,\quad  |w|_{0,D} := \norm{  w}_D.
%  %  \norm{w}_{1,D} &:= (\norm{w}^2_D+\norm{\bmnabla w}^2_D)^{\frac{1}{2}}.
%  \end{split}
%\end{displaymath}

Let $\mathcal T_h$ be a conforming and shape regular triangulation of $\Omega$. For each $T\in\mathcal T_h$, $h_T$
denotes the diameter of $T$ with  $h:=\max_{T\in\mathcal T_h}h_T$. The regularity
parameter of $\mathcal T_h$ is defined by $\rho:=\max_{T\in\mathcal T_h}h_T^d/|T|$, where $|T|$
is the $d$-dimensional Lebesgue measure of $T$. 
Let $\mathcal F_h$ denote the set of all faces of $\mathcal T_h$.

%where    $P_k(F)$ denotes the set of
%polynomials of degree $\leqslant k$ defined on $F$.
We define the mesh-dependent inner product $\bint{\cdot,\cdot}{h}$
and the corresponding norm $\norm{\cdot}_h$ as follows: for any $\lam,\mu\in L^2( \mathcal F_h )$,
\begin{equation}\label{eq:def-inner-pro}
  \bint{\lam,\mu}{h}:= \sum_{T\in\mathcal T_h}h_T\int_{\partial T}\lam\mu, \qquad \norm{\lam}_h:=\bint{\lam,\lam}{h}^{\frac{1}{2}}.
\end{equation}
We also need the following notations:
\begin{equation}
  \norm{\lam}_{h,\partial T}:=h_T^{\frac{1}{2}}\norm{\lam}_{\partial T},~\forall\lam\in L^2(\partial T),
\end{equation}
\begin{equation}\label{semi-norm}
  |\mu|_h: = (\sum_{T\in\mathcal T_h}|{\mu}|_{h, \partial T}^2)^{\frac{1}{2}},
  \quad \forall \mu \in L^2( \mathcal F_h ),
\end{equation}
where
\begin{displaymath}
  \begin{array}{rcl}
    |\mu|^2_{h,\partial T} := h_T^{-1}\norm{\mu-m_T(\mu)}^2_{\partial T},\ \
    m_T(\mu) := \frac{1}{|\partial T|}\int_{\partial T}\mu,
  \end{array}
\end{displaymath}
and $|\partial T|$ denotes the (d-1)-dimensional Hausdorff measure of $\partial T$.

Throughout this paper, $x\lesssim y~( x \gtrsim y)$ means $x\leqslant Cy~(x\geqslant Cy)$, where
$C$ denotes a positive constant that only depends on 
$d$, $k$, $\Omega$, the regularity parameter $\rho$, and the coefficient matrix $\bm a$. The notation $x \sim y$ abbreviates
$x \lesssim y\lesssim x$.

\subsection{Extended version of X-Z identity}
We start by introducing some abstract notations. Let $V$ be a finite dimensional Hilbert space
equipped with inner product $(\cdot,\cdot)$ and its induced norm $\norm{\cdot}$. Suppose $A:V\to V$ is
a linear operator which is SPD with respect to $(\cdot,\cdot)$, then $(\cdot,\cdot)_A:=(A\cdot,\cdot)$ also
defines an inner product on $V$ and we use $\norm{\cdot}_A$ to denote the corresponding norm. Let
$B:V\to V$ be a linear operator with norm
\begin{displaymath}
  \norm{B}_A:=\sup_{v\in V}\frac{\norm{Bv}_A}{\norm{v}_A}.
\end{displaymath}
\par
Suppose $V_0$, $V_1$,$V_2$,$\cdots$,$V_N$ are   finite dimensional Hilbert spaces equipped with
inner products $(\cdot,\cdot)_0$, $(\cdot,\cdot)_1$, $\cdots$, $(\cdot,\cdot)_N$ respectively.
Let $I_i:V_i\to V_0~(i=1,2,\cdots, N)$ be linear injective operators such that
\begin{displaymath}
  V_0=I_1V_1+I_2V_2+\cdots I_NV_N.
\end{displaymath}
Naturally, the adjoint operator $I_i^t$ of $I_i$ is defined  by
\begin{displaymath}
  (I_i^tv_0,v_i)_i = (v_0, I_iv_i)_0
  ~\text{ for all } v_i\in V_i\text{ and } v_0\in V_0.
\end{displaymath}
Let $\mathcal A_0:V_0\to V_0$ be SPD with respect to $(\cdot,\cdot)_0$ and define
$\mathcal A_i:V_i\to V_i$ ($i=1,2,\cdots, N$)
by $\mathcal A_i = I_i^t\mathcal A_0I_i$. Since $I_i$ is injective, $\mathcal A_i$ is SPD with
respect to $(\cdot,\cdot)_i$.  For each $i$, suppose $\mathcal R_i:V_i\to V_i$  is a good
approximation of $\mathcal A_i^{-1}$ and define the symmetrization of $\mathcal R_i$ by
\begin{equation}\label{Ri-bar}
\overline{\mathcal R_i}=\mathcal R_i^t+\mathcal R_i-\mathcal R_i^t\mathcal A_i\mathcal R_i.
\end{equation}
Then we define the operator $\mathcal B_0:V_0\to V_0$ as follows:
\par
\framebox{
  \begin{minipage}{90mm}
    For any given $b\in V_0$, $B_0b := v^{2N}$ with $v^{2N}$ defined below:\\
    $~~~~$$v^0:=0$.\\
    $~~~~${\bf for} $i=1,2,\cdots, N$\\
    $~~~~~~~~$$v^{i} := v^{i-1} + I_i\mathcal R_iI_i^t(b-\mathcal A_0v^{i-1})$;\\
    $~~~~${\bf end}\\
    $~~~~${\bf for} $i=N,N-1,\cdots, 1$\\
    $~~~~~~~~$$v^{2N-i+1} := v^{2N-i} + I_i\mathcal R^t_iI^t_i(b-\mathcal A_0v^{2N-i})$;\\
    $~~~~${\bf end}
  \end{minipage}
}
\par
Finally, following \cite{Xu;2002,Xu;2008,Chen;2010,Li;mg_wg;2014}, we are ready to present the following extended version of X-Z identity.
\begin{thm}\label{thm:X-Z identity}
  Suppose $\mathcal R_i$ is such that $\norm{I-\overline{\mathcal R_i}\mathcal A_i}_{\mathcal A_i} < 1$   for $i=1,2,\cdots, N$.
  Then it holds
  \begin{equation}\label{X-Z-Identity}
    \norm{I-\mathcal B_0\mathcal A_0}_{\mathcal A_0}=1-\frac{1}{K},
  \end{equation}
  where
  \begin{equation}
    K=\sup_{\norm{v}_{\mathcal A_0} = 1}\inf_{\sum_{i=1}^NI_iv_i=v}\sum_{i=1}^N\norm{v_i+
      \mathcal R_i^tI_i^t\mathcal A\sum_{j>i}I_jv_j}^2_{\overline{\mathcal R_i}^{-1}}.
  \end{equation}
\end{thm}
\begin{proof} The desired result can be obtained by following a similar routine  to   the proof  of (Theorem 4.1, \cite{Li;mg_wg;2014}), which is a trivial modification of the new proof \cite{Chen;2010} of the X-Z identity.
\end{proof}
\subsection{HDG framework}\label{ssec:hdg}

We  give a brief description  of the HDG framework;  One may refer to \cite{Cockburn;unified HDG;2009} for more details. 
For any $T\in\mathcal T_h$, let $V(T)\subset L^2(T)$ and
$\bm W(T)\subset [L^2(T)]^d$ be   finite dimensional spaces, $\alpha_T$ be a nonnegative
penalty function defined on $\partial T$, and $P_T^{\partial}:L^2(\partial T)\to M(\partial T)$ be
the standard $L^2$-orthogonal projection operator with
$$
M(\partial T)=\{\mu\in L^2(\partial T):\mu|_F\in P_k(F), \text{ for each face $F$ of $T$}\}.
$$
Introduce the
finite dimensional spaces
\begin{eqnarray}
					V_h &:=& \{v\in L^2(\Omega): v_h|_T \in V(T),~\forall T\in\mathcal T_h\},\\
			\bm W_h &:=& \{\bm\tau\in [L^2(\Omega)]^d: \bm\tau_h|_T \in \bm W(T),~\forall T\in\mathcal T_h\},\\
M_h&:=&\{\mu_h\in L^2( \mathcal F_h):\mu_h|_F\in P_k(F),~\forall F\in\mathcal F_h, \text{ and }
  \mu_h|_{\partial\Omega}=0\}.
\end{eqnarray}

\par
The general framework of HDG methods for the problem 
		\eqref{eq:model} reads as follows (\cite{Cockburn;unified HDG;2009}): Seek 
		$(u_h,\lambda_h,\bm\sigma_h) \in V_h \times M_h \times \bm W_h$ such that
		\begin{subequations}\label{discretization_hdg}
			\begin{eqnarray}
				(\bm C\bmsigma_h,\bm\tau_h)+(u_h,div_h\bm\tau_h)-\bndint{\lambda_h,\bm\tau_h\cdot\bm n} &=& 0\quad  \forall \bm\tau_h\in \bm W_h, \label{discrete1} \\
				-(v_h,div_h\bm\sigma_h) + \bndint{\alpha_T(P_T^{\partial}u_h - \lambda_h), v_h} &=& (f, v_h)\quad \forall v_h\in V_h, \label{discrete2}\\
				\bndint{\bm\sigma_h\cdot\bm n - \alpha_T(P_T^{\partial}u_h-\lambda_h),\mu_h} &=& 0\quad \forall \mu_h\in M_h.\label{discrete3}
			\end{eqnarray}
		\end{subequations}
	%	hold for all $(v_h,\mu_h,\bm\tau_h)\in V_h\times M_h \times \bm W_h$, 
		where 
$\bm c = \bm a^{-1}$, and $div_h$ is the broken 
		$div$ operator defined by
		$
			div_h\bm\tau_h|_T: = div(\bm\tau_h|_T)$ for any $  \bm\tau_h\in  \bm W_h, T\in\mathcal T_h.
		$
		
		Introduce the following local problem: for any $\lam\in L^2(\partial T)$, seek $(u_{\lam},\bmsigma_{\lam})\in V(T)
\times\bm W(T)$ such that
\begin{subequations}\label{eq:local hdg}
  \begin{eqnarray}
    (\bm c\bmsigma_{\lam},\bm\tau)_T+(u_{\lam},div\bm\tau)_T
    &= \bint{\lam,\bm\tau\cdot\bm n}{\partial T} & \forall \bm\tau \in \bm W(T), \\
    -(v,div\bmsigma_{\lam})_T +
    \bint{\alpha_TP_T^{\partial}u_{\lam},v}{\partial T}
    &=\bint{\alpha_T\lam,v}{\partial T} & \forall v \in V(T).
  \end{eqnarray}
\end{subequations}
%Let $a_h(\cdot,\cdot)$ be
Let $a_h(\cdot,\cdot): M_h\times M_h\rightarrow \mathbb R$ be  a  bilinear form    associated with the above local problem, defined by
\begin{equation}
  a_h(\lam_h,\mu_h) :=
  \sum_{T\in\mathcal T_h}(\bm c\bmsigma_{\lam_h},\bmsigma_{\mu_h})_T+
  \bndint{\alpha_T(P_T^{\partial}u_{\lam_h}-\lam_h),
    P_T^{\partial}u_{\mu_h}-\mu_h}.\label{eq:a_h hdg}
\end{equation}
Then  the HDG model \eqref{discretization_hdg} is equivalent to the following reduced system \cite{Cockburn;unified HDG;2009}: seek
$\lam_h\in M_h$ such that
\begin{equation}\label{eq:spd_sys}
  a_h(\lam_h,\mu_h)=(f,u_{\mu_h})_{\Omega},~\forall\mu_h\in M_h.
\end{equation}
%\begin{rem}
  We note that once the Lagrangian multiplier approximation $\lam_h$ is resolved, the numerical flux $\bm\sigma_h$ and the potential approximation $u_h$ can be obtained
  in an element-by-element fashion by \eqref{eq:local hdg}. %We refer the reader to \cite{Cockburn;unified HDG;2009}  for more details.
%\end{rem}

\section{Two-level algorithm}\label{sec:mg}

We recall  that the triangulation $\mathcal T_h$ is assumed to be conforming and shape regular. In addition, we assume the   regularity estimate \eqref{eq:reg} holds  with  $\alpha\in [0,1]$.

For the sake of convenience, in the rest of this paper we shall use the notation $(\cdot,\cdot)$ to abbreviate  the
$L^2$-inner product $(\cdot,\cdot)_\Omega$.

\subsection{Algorithm description}\label{Algorithm}

At first, we introduce the $H^1$-conforming piecewise linear finite element space
\begin{equation}
  V_h:=\{v_h\in H_0^1(\Omega):v_h|_T\in P_1(T)\}.
\end{equation}
%and some operators as follows.
We then define the prolongation operator $I_h:V_h\to M_h$ and its adjoint operator $I_h^t:M_h\to V_h$ respectively by
\begin{eqnarray}
%  \begin{array}{lcll}
\label{De-I_h}
    \inpro{I_hv_h,\mu_h} &=& \inpro{v_h,\mu_h}\quad \forall v_h\in V_h,\\
    \label{De-I_hT}
    (I_h^t\mu_h,v_h)     &=& \inpro{\mu_h,I_hv_h}\quad \forall \mu_h\in M_h,
%  \end{array}
\end{eqnarray}
%for all $ $ and $ $.
and define the operators $\mathcal A_h:M_h\to M_h$ and $\we{\mathcal A_h}:V_h\to V_h$ respectively by
\begin{eqnarray}
%  \begin{array}{lcl}
\label{De-A_h}
    \inpro{\mathcal A_h\lam_h,\mu_h} &=& a_h(\lam_h,\mu_h)\quad \forall \lam_h,\mu_h\in M_h,\\
    \label{De-A_h-We}
    (\we{\mathcal A_h}u_h,v_h)       &=& (\bm a\bmnabla u_h,\bmnabla v_h) \quad \forall u_h,v_h\in V_h.
  %\end{array}
\end{eqnarray}
%for all $ $, and $ $.
Let $\mathcal R_h:M_h\to M_h$ and $\we{\mathcal R_h}:V_h\to V_h$ be good approximations of
$\mathcal A_h^{-1}$ and $\we{\mathcal A_h}^{-1}$ respectively, with $\mathcal R_h^t$ and
$\we{\mathcal R_h}^t$ satisfying
\begin{displaymath}
  \begin{array}{lcll}
    \inpro{\mathcal R_h^t\lam_h,\mu_h} &=& \inpro{\lam_h,\mathcal R_h\mu_h}\quad \forall \lam_h,\mu_h\in M_h,\\
  (\we{\mathcal R_h}^tu_h,v_h)&=&(u_h,\we{\mathcal R_h}v_h)\quad \forall u_h,v_h\in V_h.
  \end{array}
\end{displaymath}
%for all $\lam_h,\mu_h\in M_h$, and $u_h,v_h\in V_h$.
Finally  we define the operator $\mathcal B_h:M_h\to M_h$ as follows:\\

% \begin{center}
\framebox{
  \begin{minipage}{85mm}
    For any $\eta_h\in M_h$, $\mathcal B_h\eta_h:=\mu_h^4$ with $\mu_h^4$ defined below:\\
    $~~~~$1. $\mu_h^1:=\mathcal R_h\eta_h$;\\
    $~~~~$2. $\mu_h^2:=\mu_h^1+I_h\we{\mathcal R_h}I_h^t(\eta_h-\mathcal A_h\mu_h^1)$;\\
    $~~~~$3. $\mu_h^3:=\mu_h^2+I_h\we{\mathcal R_h}^tI_h^t(\eta_h-\mathcal A_h\mu_h^2)$;\\
    $~~~~$4. $\mu_h^4:=\mu_h^3+\mathcal R_h^t(\eta_h-\mathcal A_h\mu_h^3)$.
    % \end{algorithm}
  \end{minipage}}
% \end{center}

In view of the operators $\mathcal A_h$ and $\mathcal B_h$,  we   present the following two-level algorithm for the system \eqref{eq:spd_sys}:

\framebox{
  \begin{minipage}{120mm}
    \begin{algorithm}\label{algo:two-level}
      Let $b_h\in M_h$ be given. We solve the equation $\mathcal A_h\lam_h=b_h$ below:\\
      $~~~~\lam_h^0=0$,\\
      $~~~~${\bf for} $j=1,2,\cdots$\\
      $~~~~~~~~\lam_h^j:=\lam_h^{j-1}+\mathcal B_h(b_h-\mathcal A_h\lam_h^{j-1})$;\\
      $~~~~${\bf end}
    \end{algorithm}
  \end{minipage}}

\subsection{Main results}\label{main-results}

We first  introduce the following symmetrizations of $\mathcal R_h$ and $\we{\mathcal R_h}$:
\begin{eqnarray}
  \overline{\mathcal R_h}
  &:=& \mathcal R_h^t+\mathcal R_h-\mathcal R_h^t\mathcal A_h\mathcal R_h,
       \label{def_sysm_R}\\
  \overline{\we{\mathcal R_h}}
  &:=& \we{\mathcal R_h}^t+\we{\mathcal R_h}-\we{\mathcal R_h}^t\we{\mathcal A_h}\we{\mathcal R_h}
       \label{def_sysm_we_R},\\
  \overline{\we{\we{\mathcal R_h}}}
  &:=& \we{\mathcal R_h}^t+\we{\mathcal R_h}-\we{\mathcal R_h}^t\we{\we{\mathcal A_h}}\we{\mathcal R_h},
       \label{def_sysm_we_we_R}
\end{eqnarray}
where
\begin{equation}\label{we-we-A-h}
  \we{\we{\mathcal A_h}}:= I_h^t\mathcal A_hI_h.
\end{equation}
Then we present  some assumptions below.
\begin{assumption}\label{assum:a_h}
  For any $\lam_h\in M_h$, it holds
  \begin{equation}
    \norm{\bm c^{\frac{1}{2}}\bmsigma_{\lam_h}}^2_T +
    \norm{\alpha_T^{\frac{1}{2}}(P_T^{\partial}u_{\lam_h}-\lam_h)}^2_{\partial T}
    \sim\vertiii{\lam_h}^2_{h,\partial T},~\forall T\in\mathcal T_h.
  \end{equation}
\end{assumption}
\begin{assumption}\label{assum:M_h N_h}
 It holds
  \begin{equation}\label{constr-tilde-Rh}
    \sqrt{\frac{1+\mathcal N_h}{1-\mathcal N_h}}
    \bigg[(1+\mathcal N_h)\mathcal M_h+\mathcal N_h\bigg] < 1,
  \end{equation}
  where
  \begin{eqnarray}\label{MN_h}
    \mathcal M_h:=\norm{I-\we{\mathcal R_h}\we{\mathcal A_h}}_{\we{\mathcal A_h}},\quad 
    \mathcal N_h:=\norm{I-\we{\mathcal A_h}^{-1}\we{\we{\mathcal A_h}}}_{\we{\mathcal A_h}}.
  \end{eqnarray}
\end{assumption}
\begin{assumption}\label{assum:R_h 1}
  Let $\mathcal R_h:M_h\to M_h$ be SPD with respect to
  $\inpro{\cdot,\cdot}$ such that
  \begin{eqnarray}
    &0<\sigma(\mathcal R_h\mathcal A_h)<\omega,&\label{eq:assum R_hA_h}\\
    &\norm{\lam_h}_{\overline{\mathcal R_h}^{-1}}\lesssim
      \bigg(\sum_{T\in\mathcal T_h}h_T^{-2}\norm{\lam_h}^2_{h,\partial T}\bigg)^{\frac{1}{2}},
      ~\forall\lam_h\in M_h,&\label{eq:assum inv R_h}
  \end{eqnarray}
  where $\sigma(\mathcal R_h\mathcal A_h)$ denotes the set of all eigenvalues of
  $\mathcal R_h\mathcal A_h$, and $ \omega$ is a  constant  with $0<\omega<2$.
\end{assumption}
\begin{rem}
  Obviously, {\bf Assumption \ref{assum:a_h}} implies
  $$
  a_h(\lam_h,\lam_h)\sim\vertiii{\lam_h}^2_h,~\forall\lam_h\in M_h,
  $$
  and hence $\mathcal A_h$ is SPD with respect to $\inpro{\cdot,\cdot}$.
  Since $I_h:V_h\to M_h$ is injective, it is easy to verify that
  $\we{\we{\mathcal A_h}}$ is SPD with respect to $(\cdot,\cdot)$.
  What's more, by a simple estimate $\vertiii{\lam_h}_{h,\partial T}
  \lesssim h_T^{-1}\norm{\lam_h}_{h,\partial T}$, {\bf Assumption \ref{assum:a_h}} implies that the largest
  eigenvalue of $\mathcal A_h$ satisfies
  $\lam_{max}(\mathcal A_h)\lesssim h^{-2}$ under the condition that
  $\mathcal T_h$ is quasi-uniform.
\end{rem}

\begin{rem} \label{rem32} In  {\bf Assumption \ref{assum:M_h N_h}}, when $\mathcal N_h$ is given,  the condition \eqref{constr-tilde-Rh}  requires that $\mathcal M_h$ is sufficiently small, i.e. the operator $\we{\mathcal R_h}$ is a good-enough approximation of $\we{\mathcal A_h}^{-1}$. Fortunately, for the $H^1$-conforming linear element approximation $\we{\mathcal A_h}$, the research of the choice of $\we{\mathcal R_h}$ is mature.  
 As will be shown in Section \ref{sec:appl} for some applications, it holds $\mathcal N_h=0$
  or $\mathcal N_h\lesssim h$. In the former case,  \eqref{constr-tilde-Rh}  is reduced to  the constraint 
  \begin{equation}\label{M-h-cons} 
  \mathcal M_h<1.
  \end{equation}
  In the latter case, $h$ should be also small
  enough to ensure \eqref{constr-tilde-Rh}. We note that  {\bf Assumption \ref{assum:M_h N_h}} requires implicitly the constraint
   \eqref{M-h-cons}.    
\end{rem}

\begin{rem}\label{Rem3.3}
  It is evident that the condition \eqref{eq:assum R_hA_h} in {\bf Assumption \ref{assum:R_h 1}} implies that
  $\norm{I-\mathcal R_h\mathcal A_h}_{\mathcal A_h}<1$, which means
  $$
  \norm{I-\overline{\mathcal R_h}\mathcal A_h}_{\mathcal A_h}
  =\norm{I-\mathcal R_h\mathcal A_h}^2_{\mathcal A_h} < 1.
  $$
  Suppose {\bf Assumption \ref{assum:a_h}} is true.
  If we choose the Richardson iteration as $\mathcal R_h$, i.e.
  $\mathcal R_h=\frac{1}{\lam_{max}(\mathcal A_h)}I$,
  then \eqref{eq:assum R_hA_h} holds with $\omega=1$, while \eqref{eq:assum inv R_h}
  holds only in the case that $\mathcal T_h$ is quasi-uniform. However, if we set
  $\mathcal R_h$ to be the symmetric Gauss-Seidel iteration, then \eqref{eq:assum R_hA_h}
  holds with $\omega=1$, and \eqref{eq:assum inv R_h} holds as long as $\mathcal T_h$ is
  conforming and shape regular.  We refer to {\bf Appendix A} for a concise  analysis of the symmetric Gauss-Seidel iteration.
\end{rem}

We state  the main results in two theorems  below.
\begin{thm}\label{thm:convergence}
  Under {\bf Assumptions \ref{assum:a_h}-\ref{assum:R_h 1}},
  it holds
  \begin{equation}\label{eq:convergence}
    \norm{(I-\mathcal B_h\mathcal A_h)}_{\mathcal A_h}=1-\frac{1}{K},
  \end{equation}
  where
  \begin{equation}\label{eq:ubnd_K}
    K \lesssim 1 + \mathcal N_h +
    \frac
    {
      (1+\mathcal N_h)/(2-\omega)
    }
    {
      1-\frac{1+\mathcal N_h}{1-\mathcal N_h}
      \bigg[(1+\mathcal N_h)\mathcal M_h+\mathcal N_h\bigg]^2
    }.
  \end{equation}
\end{thm}

%Out of its importance, we shall consider one smoother excluded from {\bf Assumption \ref{assum:R_h 1}}: $\mathcal R_h$ denotes one sweep
%of Gauss-Seidel iteration.
\begin{thm}\label{thm:convergence2}
  Let $\mathcal R_h$ be one sweep of Gauss-Seidel iteration. Then,
  Under {\bf Assumptions \ref{assum:a_h}-\ref{assum:M_h N_h}}, the relation
  \eqref{eq:convergence}   holds with $\omega=1$.
\end{thm}

We shall prove these two theorems in Section \ref{conver-anal}.
\begin{rem}
  Since we only assume $\mathcal T_h$ to be conforming and shape regular,
  it's important that Theorems \ref{thm:convergence}-\ref{thm:convergence2}
   hold on non-quasi-uniform  grids, as long as there is a proper choice of $\we{\mathcal R_h}$ for the $H^1$-conforming linear element approximation. We refer the reader to \cite{Brandt;1977, McCormick;1984,
    McCormick;1986, Dahmen;1992,Bornemann;1993,Aksoylu;2006,Wu;2006,Chen;2011} for the construction of $\we{\mathcal R_h}$ on
  adaptive grids, and to \cite{Brandt;1982, Brandt;1986, Vanek;1996, Livne;2004} for 
    the construction  on completely unstructured grids.
\end{rem}
% \begin{rem}
%   Theorems \ref{thm:convergence}-\ref{thm:convergence2} imply that
%   $\mathcal B_h$ can also be an optimal preconditioner for $\mathcal A_h$.
%   Although \cite{Chen;2014} provided two optimal auxiliary space preconditioners
%   for two type of WG methods and \cite{Li;bpx;2014} showed the optimality of
%   the BPX preconditioner for a large class of nonstandard finite element methods,
%   they both required $\mathcal T_h$ to be quasi-uniform.
%  %   Since for some HDG methods, it may not be easy
%  %   to ensure {\bf Assumption \ref{assum:M_h N_h}}, we shall analyze
%  %   two robust preconditioners without {\bf Assumption \ref{assum:M_h N_h}}
%  %   in the appendix.
% \end{rem}

\subsection{Convergence analysis}\label{conver-anal}
\subsubsection{Proof of Theorem \ref{thm:convergence}}
\begin{lem}\label{lem:WeWeA}
  Under {\bf Assumptions \ref{assum:a_h}-\ref{assum:M_h N_h}}, it holds
  \begin{equation}\label{eq:WeWeA}
    \norm{I-\overline{\we{\we{\mathcal R_h}}}\we{\we{\mathcal A_h}}}_{\we{\we{\mathcal A_h}}}<1,
  \end{equation}
  \begin{equation}\label{eq:v_h WeWeR_h}
    \norm{v_h}^2_{\overline{\we{\we{\mathcal R_h}}}^{-1}}
    \leqslant\frac {1+\mathcal N_h}
    {1-\frac{1+\mathcal N_h}{1-\mathcal N_h}\bigg[(1+\mathcal N_h)\mathcal M_h+\mathcal N_h\bigg]^2}
    \norm{v_h}^2_{\we{\mathcal A_h}},~\forall v_h\in V_h.
  \end{equation}
\end{lem}
\begin{proof}
  Since $\we{\mathcal A_h}^{-1}\we{\we{\mathcal A_h}}$ is symmetric with respect to
  $(\cdot,\cdot)_{\we{\mathcal A_h}}$, we have
  \begin{equation}\label{eq:s v_h}
    (1-\mathcal N_h)\norm{v_h}^2_{\we{\mathcal A_h}}
    \leqslant \norm{v_h}^2_{\we{\we{\mathcal A_h}}}\leqslant
    (1+\mathcal N_h)\norm{v_h}^2_{\we{\mathcal A_h}},~\forall v_h\in V_h.
  \end{equation}
  For any linear operator $\mathcal S_h:V_h\to V_h$, it holds
  \begin{equation}\label{eq:s S_h}
    \begin{split}
      \norm{\mathcal S_h}_{\we{\we{\mathcal A_h}}}
      &=\sup_{v_h\in V_h}\frac{\norm{\mathcal S_hv_h}_{\we{\we{\mathcal A_h}}}}
      {\norm{v_h}_{\we{\we{\mathcal A_h}}}}\\
      &\leqslant\sqrt{\frac{1+\mathcal N_h}{1-\mathcal N_h}}\sup_{v_h\in V_h}
      \frac{\norm{\mathcal S_hv_h}_{\we{\mathcal A_h}}}{\norm{v_h}_{\we{\mathcal A_h}}}\\
      &=\sqrt{\frac{1+\mathcal N_h}{1-\mathcal N_h}}\norm{\mathcal S_h}_{\we{\mathcal A_h}}.
    \end{split}
  \end{equation}
 Then,  from
  \begin{displaymath}
    \begin{split}
      \norm{I-\we{\mathcal R_h}\we{\we{\mathcal A_h}}}_{\we{\mathcal A_h}}
      &\leqslant\norm{I-\we{\mathcal R_h}\we{\mathcal A_h}}_{\we{\mathcal A_h}}+
      \norm{\we{\mathcal R_h}\we{\mathcal A_h}(I-\we{\mathcal A_h}^{-1}\we{\we{\mathcal A_h}})}_{\we{\mathcal A_h}}\\
      &\leqslant\mathcal M_h+\mathcal N_h\norm{\we{\mathcal R_h}\we{\mathcal A_h}}_{\we{\mathcal A_h}}\\
      &\leqslant(1+\mathcal N_h)\mathcal M_h+\mathcal N_h
    \end{split}
  \end{displaymath}
  it follows
  \begin{equation}\label{eq:lxy}
    \begin{split}
      \norm{I-\overline{\we{\we{\mathcal R_h}}}\we{\we{\mathcal A_h}}}_{\we{\we{\mathcal A_h}}}
      &=\norm{I-\we{\mathcal R_h}\we{\we{\mathcal A_h}}}_{\we{\we{\mathcal A_h}}}^2\\
      &\leqslant\frac{1+\mathcal N_h}{1-\mathcal N_h}
      \norm{I-\we{\mathcal R_h}\we{\we{\mathcal A_h}}}^2_{\we{\mathcal A_h}}\qquad\text{(by \eqref{eq:s S_h})}\\
      &\leqslant\frac{1+\mathcal N_h}{1-\mathcal N_h}\bigg[(1+\mathcal N_h)\mathcal M_h
      +\mathcal N_h\bigg]^2,
    \end{split}
  \end{equation}
  which  immediately implies \eqref{eq:WeWeA}.

  On the other hand, by the definition of $\overline{\we{\we{\mathcal R_h}}}$, we can get
  \begin{displaymath}
    \norm{I-\overline{\we{\we{\mathcal R_h}}}\we{\we{\mathcal A_h}}}_{\we{\we{\mathcal A_h}}}
    = 1-\lam_{min}(\overline{\we{\we{\mathcal R_h}}}\we{\we{\mathcal A_h}}),
  \end{displaymath}
  where $\lam_{min}(\overline{\we{\we{\mathcal R_h}}}\we{\we{\mathcal A_h}})$ denotes the
  smallest eigenvalue of $\overline{\we{\we{\mathcal R_h}}}\we{\we{\mathcal A_h}}$.  The above relation, together with the fact  that, due to
  \eqref{eq:WeWeA},   $\overline{\we{\we{\mathcal R_h}}}$ is SPD with respect to
  $(\cdot,\cdot)$, yields
   \begin{equation}
    \begin{split}
      \norm{v_h}^2_{\overline{\we{\we{\mathcal R_h}}}^{-1}}
      &\leqslant\frac{1}{\lam_{min}(\overline{\we{\we{\mathcal R_h}}}\we{\we{\mathcal A_h}})}
      \norm{v_h}^2_{\we{\we{\mathcal A_h}}}\\
      &=\frac{1}
      {1-\norm{I-\overline{\we{\we{\mathcal R_h}}}\we{\we{\mathcal A_h}}}_{\we{\we{\mathcal A_h}}}}
      \norm{v_h}^2_{\we{\we{\mathcal A_h}}}.
    \end{split}
  \end{equation}
  Finally, the desired inequality \eqref{eq:v_h WeWeR_h} follows immediately from \eqref{eq:s v_h} and \eqref{eq:lxy}.
  This completes the proof.
\end{proof}

\begin{lem}\label{lem:appl X-Z 1}
  Under {\bf Assumptions \ref{assum:a_h}-\ref{assum:R_h 1}}, the relation \eqref{eq:convergence}
  holds with
  \begin{equation}
    K=\sup_{\norm{\lam_h}_{\mathcal A_h}=1}\inf_{\mu_h+I_hv_h=\lam_h}
    \norm{\mu_h+\mathcal R_h\mathcal A_hI_hv_h}^2_{\overline{\mathcal R_h}^{-1}}+
    \norm{v_h}^2_{\overline{\we{\we{\mathcal R_h}}}^{-1}}.
  \end{equation}
\end{lem}
\begin{proof}
The conclusion follows from   the space decomposition $M_h=M_h+I_hV_h$ and the extended version of X-Z identity \eqref{X-Z-Identity}. 
\end{proof}

\begin{lem}\label{lem:R_hA_hI_hv_h}
  Under {\bf Assumptions \ref{assum:a_h}-\ref{assum:R_h 1}}, it holds
  \begin{equation}\label{eq:R_hA_hI_hv_h}
    \norm{\mathcal R_h\mathcal A_hI_hv_h}^2_{\overline{\mathcal R_h}^{-1}}\leqslant
    \max\bigg\{1,\frac{\omega}{2-\omega}\bigg\}(1+\mathcal N_h)\norm{v_h}^2_{\we{\mathcal A_h}},~\forall v_h\in V_h.
  \end{equation}
\end{lem}
\begin{proof}
  Denote $\mathcal S_h:= \mathcal R_h\mathcal A_h$.
  % It follows from \eqref{eq:assum R_hA_h} that
  % $\sigma(\mathcal S_h)\subset (0,1]$, where $\sigma(\mathcal S_h)$ denotes the set of all
  % eigenvalues of $\mathcal S_h$.
  Noting that
  \begin{displaymath}
    \overline{\mathcal R_h} = 2\mathcal R_h-\mathcal R_h\mathcal A_h\mathcal R_h
    =(2\mathcal S_h-\mathcal S_h^2)\mathcal A_h^{-1},
  \end{displaymath}
  we have
  \begin{equation}\label{eq:s R_hA_hI_hv_h}
    \begin{split}
      \norm{\mathcal R_h\mathcal A_hI_hv_h}^2_{\overline{\mathcal R_h}^{-1}}
      &=\inpro{\overline{\mathcal R_h}^{-1}\mathcal R_h\mathcal A_hI_hv_h,
        \mathcal R_h\mathcal A_hI_hv_h}\\
      &=\inpro{\mathcal A_h(2\mathcal S_h-\mathcal S_h^2)^{-1}\mathcal S_hI_hv_h,
        \mathcal R_h\mathcal A_hI_hv_h}\\
      &=\ninpro{\mathcal S_h(2\mathcal S_h-\mathcal S_h^2)^{-1}\mathcal S_hI_hv_h,I_hv_h}.
    \end{split}
  \end{equation}
  Since $\mathcal S_h$ is SPD with respect to $\ninpro{\cdot,\cdot}$ and the inequality
  \begin{displaymath}
    t(2t-t^2)^{-1}t\leqslant\max\bigg\{1,\frac{\omega}{2-\omega}\bigg\},t\in (0,\omega]
  \end{displaymath}
 holds, the relation \eqref{eq:s R_hA_hI_hv_h}, together with \eqref{eq:s v_h}, immediately yields the desired
  estimate \eqref{eq:R_hA_hI_hv_h}.
\end{proof}
From Lemmas \ref{lem:WeWeA}-\ref{lem:R_hA_hI_hv_h} and {\bf Assumption \ref{assum:R_h 1}},
we obtain immediately the lemma below.
\begin{lem}\label{lem:last}
  Under {\bf Assumptions \ref{assum:a_h}-\ref{assum:R_h 1}}, the relation \eqref{eq:convergence} holds with
  \begin{equation}\label{eq:last K}
    K\lesssim
    \sup_{\norm{\lam_h}_{\mathcal A_h}=1}\inf_{\mu_h+I_hv_h=\lam_h}
    \sum_{T\in\mathcal T_h}h_T^{-2}\norm{\mu_h}^2_{h,\partial T} +
    \frac
    {
      (1+\mathcal N_h)/(2-\omega)
    }
    {
      1-
      \frac
      {
        1+\mathcal N_h
      }
      {
        1-\mathcal N_h
      }
      \bigg[
      (1+\mathcal N_h)\mathcal M_h+\mathcal N_h
      \bigg]^2
    }
    \norm{v_h}^2_{\we{\mathcal A_h}}.
  \end{equation}
\end{lem}

To further derive \eqref{eq:ubnd_K},
we introduce the operator $P_h:M_h\to V_h$ with \begin{displaymath}
  P_h\lam_h(\bm x) =
  \left\{
    \begin{array}{rcll}
      % P_h\lam_h(\bm x) &=&
                             \frac{\sum\limits_{T\in\omega_{\bm x}}m_T(\lam_h)}
                             {\sum\limits_{T\in\omega_{\bm x}}1}, &
                                                                    \text{if $\bm x$ is an interior node of $\mathcal T_h$,}\\
      0, &\text{if $\bm x\in\partial\Omega$}
    \end{array}
  \right.
\end{displaymath}
 for any $\lam_h\in M_h$,
 where $\omega_{\bm x}$ denotes the set $\{T\in\mathcal T_h:\bm x \text{ is a vertex of $T$}\}$.
We have the following important estimates for $P_h$.  
\begin{lem}\label{lem:P_h}
  For any $\lam_h\in M_h$, it holds
  \begin{eqnarray}
    |P_h\lam_h|_{1,\Omega}&\lesssim&\vertiii{\lam_h}_h,\label{eq:esti P_h}\\
    \bigg(\sum_Th_T^{-2}\norm{(I-I_hP_h)\lam_h}^2_{h,\partial T}\bigg)^{\frac{1}{2}}&\lesssim&\vertiii{\lam_h}_h.
                                                                                               \label{eq:esti I-I_hP_h}
  \end{eqnarray}
\end{lem}
\begin{proof}
  For each $T\in \mathcal T_h$, we denote $\omega_T:=\{T'\in\mathcal T_h:\text{ $T'$ and $T$ share a vertex}\}$
  and use $\mathcal N(T)$ to denote  the set of all vertexes of $T$.  Assume all vertexes of $T$ are interior nodes of $\mathcal T_h$, then it holds
  \begin{eqnarray}\label{eq:s P_h}
%    \begin{split}
      \norm{P_h\lam_h-m_T(\lam_h)}^2_{\partial T}
      &\lesssim h_T^{d-1}\sum_{\bm x\in\mathcal N(T)}|P_h\lam_h(\bm x)-m_T(\lam_h)|^2\nonumber\\
      &\lesssim h_T^{d-1}\sum_{\bm x\in\mathcal N(T)}\sum_{\substack{T_1,T_2\in\omega_{\bm x}\nonumber\\
          T_1,T_2\text{ share a same face }}}|m_{T_1}(\lam_h)-m_{T_2}(\lam_h)|^2\nonumber\\
      &\lesssim \sum_{\bm x\in\mathcal N(T)}
      \sum_{\substack{T_1,T_2\in\omega_{\bm x}\\ T_1,T_2\text{ share a same face }}}
      \norm{m_{T_1}(\lam_h)-m_{T_2}(\lam_h)}^2_{\partial T_1 \cap\partial T_2}\nonumber\\
      &\lesssim \sum_{T'\in\omega_T}h_{T'}\vertiii{\lam_h}^2_{h,\partial T'}.
   % \end{split}
  \end{eqnarray}
%  i. e.,
%  \begin{equation}\label{eq:s P_h}
%    \norm{P_h\lam_h-m_T(\lam_h)}^2_{\partial T}\lesssim \sum_{T'\in\omega_T}h_{T'}\vertiii{\lam_h}^2_{h,\partial T'}.
%  \end{equation}
 Similarly, we can show by  a trivial modification that  \eqref{eq:s P_h} also holds
  in the case that there is a vertex of $T$ that belongs to $\partial\Omega$. 
 As a result,  the estimate \eqref{eq:esti P_h} follows from
  \begin{displaymath}
    \begin{array}{lcll}
      |P_h\lam_h|^2_{1,T}
      &=&|P_h\lam_h-m_T(\lam_h)|^2_{1,T}&\\
      &\lesssim& h_T^{-2}\norm{P_h\lam_h-m_T(\lam_h)}^2_T & \text{(by inverse estimate)}\\
      &\lesssim& h_T^{-1}\norm{P_h\lam_h-m_T(\lam_h)}^2_{\partial T} & \\
      &\lesssim& \sum_{T'\in\omega_T}\vertiii{\lam_h}^2_{h,\partial T'}.&\text{(by \eqref{eq:s P_h})}
    \end{array}
  \end{displaymath}
  \par
 
 On the other hand,  from
  \begin{displaymath}
    \begin{array}{rcll}
      h_T\norm{I_hP_h\lam_h-m_T(\lam_h)}^2_{\partial T}
      &\leqslant& h_T\norm{P_h\lam_h-m_T(\lam_h)}^2_{\partial T} & \\
      &\lesssim& \sum_{T'\in\omega_T}h^2_{T'}\vertiii{\lam_h}^2_{h,\partial T'} & \text{(by \eqref{eq:s P_h})}
    \end{array}
  \end{displaymath}
 it follows  
  \begin{displaymath}
    \begin{split}
      h_T\norm{(I-I_hP_h)\lam_h}^2_{\partial T}
      &\lesssim h_T\norm{\lam_h-m_T(\lam_h)}^2_{\partial T} + h_T\norm{I_hP_h\lam_h-m_T(\lam_h)}^2_{\partial T}\\
      &\lesssim\sum_{T'\in\omega_T}h^2_{T'}\vertiii{\lam_h}^2_{h,\partial T'},
    \end{split}
  \end{displaymath}
  which indicates \eqref{eq:esti I-I_hP_h} immediately. 
  \end{proof}

\begin{rem}
Although similar estimates
were presented in \cite{Li;mg_wg;2014,Li;bpx;2014} for quasi-uniform grids, the estimates in  Lemma \ref{lem:P_h} are sharper
in the sense that $\mathcal T_h$ here is not assumed to be quasi-uniform.
\end{rem}

\par
Finally, we are in a position to prove Theorem \ref{thm:convergence}.\\
{\bf Proof of Theorem \ref{thm:convergence}}.
For any $\lam_h\in M_h$, set $\mu_h: = \lam_h-I_hP_h\lam_h$ and $v_h:=P_h\lam_h$.  Using
Lemma \ref{lem:P_h}, we have
\begin{displaymath}
  \begin{split}
    &\sum_{T\in\mathcal T_h}h_T^{-2}\norm{\mu_h}^2_{h,\partial T}+
    \frac
    {
      (1+\mathcal N_h)/(2-\omega)
    }
    {
      1-\frac{1+\mathcal N_h}{1-\mathcal N_h}{\bigg[(1+\mathcal N_h)\mathcal M_h+\mathcal N_h\bigg]^2}
    }
    \norm{v_h}^2_{\we{\mathcal A_h}}\\
    \lesssim&
    \left\{
      1+
      \frac
      {
        (1+\mathcal N_h)/(2-\omega)
      }
      {
        1-\frac{1+\mathcal N_h}{1-\mathcal N_h}{\bigg[(1+\mathcal N_h)\mathcal M_h+\mathcal N_h\bigg]^2}
      }
    \right\}
    \norm{\lam_h}^2_{\mathcal A_h},
  \end{split}
\end{displaymath}
which implies
\begin{displaymath}
  \begin{split}
    &\sup_{\norm{\lam_h}_{\mathcal A_h}=1}\inf_{\mu_h+I_hv_h=\lam_h}
    \sum_{T\in\mathcal T_h}h_T^{-2}\norm{\mu_h}^2_{h,\partial T}+
    \frac
    {
      (1+\mathcal N_h)/(2-\omega)
    }
    {
      1-\frac{1+\mathcal N_h}{1-\mathcal N_h}\bigg[(1+\mathcal N_h)\mathcal M_h+\mathcal N_h\bigg]^2
    }
    \norm{v_h}^2_{\we{\mathcal A_h}}\\
    \lesssim&
    1+
    \frac
    {
      (1+\mathcal N_h)/(2-\omega)
    }
    {
      1-\frac{1+\mathcal N_h}{1-\mathcal N_h}\bigg[(1+\mathcal N_h)\mathcal M_h+\mathcal N_h\bigg]^2
    }.
  \end{split}
\end{displaymath} 
Then  Theorem \ref{thm:convergence} follows from Lemma
\ref{lem:last} immediately.

\subsubsection{Proof of Theorem \ref{thm:convergence2}}
Let $\{\eta_i:i=1,2,\ldots, N\}$ be the standard nodal basis for $M_h$. We have the following space
decomposition: 
$$M_h=\text{span}\{\eta_1\}+\text{span}\{\eta_2\}+\ldots+\text{span}\{\eta_N\}+I_hV_h.$$
Define
$P_i:M_h\to\text{span}\{\eta_i\}$ by $\ninpro{P_i\lam_h,\eta_i} = \ninpro{\lam_h,\eta_i}$
for $i=1,2,\ldots, N$. Then, by the extended version of X-Z identity \eqref{X-Z-Identity},
we have the following lemma.
\begin{lem}\label{lem:appl X-Z 2}
  Under the conditions of Theorem \ref{thm:convergence2}, the relation
  \eqref{eq:convergence} holds with
  \begin{equation}
    K=\sup_{\norm{\lam_h}_{\mathcal A_h}=1}\inf_{\sum_{i=1}^N\mu_i+I_hv_h=\lam_h}\sum_{i=1}^N
    \norm{\mu_i+P_i\bigg(\sum_{j>i}\mu_j+I_hv_h\bigg)}^2_{\mathcal A_h} +
    \norm{v_h}^2_{\overline{\we{\we{\mathcal R_h}}}^{-1}}.
  \end{equation}
\end{lem}
\begin{lem}\label{lem:esti decom 2}
  Under {\bf Assumptions \ref{assum:a_h}-\ref{assum:M_h N_h}},  it holds
  \begin{equation}\label{eq:esti decom 2}
    \sum_{i=1}^N\norm{\mu_i+P_i\bigg(\sum_{j>i}\mu_j+I_hv_h\bigg)}^2_{\mathcal A_h}
    \lesssim\sum_{T\in\mathcal T_h}h_T^{-2}\norm{\mu_h}^2_{h,\partial T}+
    (1+\mathcal N_h)\norm{v_h}^2_{\we{\mathcal A_h}}
  \end{equation}
  for any $v_h\in V_h$ and 
  $\mu_i\in\text{span}\{\eta_i\}$ ($i=1,2,\ldots,N$)    with $\mu_h=\sum_{i=1}^N\mu_i$.
\end{lem}
\begin{proof}
  Define
  $\Sigma_i:=\{T:\text{there exists one face $F$ of $T$ such that $\eta_i|_F\not=0$}\}$
  and $\omega_i:=\cup_{T\in\Sigma_i}T$. Apparently, for any given $1\leqslant i_0\leqslant N$,
  there are at most $J$ of $\{\omega_i\}$, $\{\omega_{i_j}:j=1,2,\ldots, J\}$, such that
  $\omega_{i_0}\cap\omega_{i_j}\not=\phi$ ($j=1,2,\ldots, J$), where $J$ only depends on the dimension number $d$ and the shape regularity parameter
  $\rho$. 
  
   It is easy to verify 
  \begin{equation}\label{eq:s P_iI_h}
    \begin{split}
      \sum_{i=1}^N\norm{P_iI_hv_h}^2_{\mathcal A_h}
      &\lesssim\norm{I_hv_h}^2_{\mathcal A_h}\\
      &\lesssim(1+\mathcal N_h)\norm{v_h}^2_{\we{\mathcal A_h}}\qquad\text{(by \eqref{eq:s v_h})}
    \end{split}
  \end{equation}
  and
  \begin{displaymath}
    \begin{split}
      \sum_{i=1}^N\norm{\mu_i+P_i\sum_{j>i}\mu_j}^2_{\mathcal A_h}= \sum_{i=1}^N\norm{\mu_i+P_i\sum_{j>i,\omega_i\cap\omega_j\not=\phi}\mu_j}^2_{\mathcal A_h}
      &\lesssim\sum_{i=1}^N\norm{\mu_i}^2_{\mathcal A_h}\\
      &\lesssim\sum_{i=1}^N\vertiii{\mu_i}^2_h.\qquad\text{(by {\bf Assumption \ref{assum:a_h}})}
    \end{split}
  \end{displaymath}
  Then, from
  \begin{equation}
    \vertiii{\mu}_{h,\partial T}\lesssim h_T^{-1}\norm{\mu}_{h,\partial T},
    ~\forall\mu\in M(\partial T)
  \end{equation}
  it follows
  \begin{equation}\label{eq:s mu_i}
    \begin{split}
      \sum_{i=1}^N\norm{\mu_i+P_i\sum_{j>i}\mu_j}^2_{\mathcal A_h}
      &\lesssim\sum_{i=1}^N\sum_{T\in\mathcal T_h}h_T^{-2}\norm{\mu_i}^2_{h,\partial T}\\
      &\lesssim\sum_{T\in\mathcal T_h}h_T^{-2}\sum_{i=1}^N\norm{\mu_i}^2_{h,\partial T}\\
      &\lesssim\sum_{T\in\mathcal T_h}h_T^{-2}\norm{\mu_h}^2_{h,\partial T},
    \end{split}
  \end{equation}
 where we have used the estimate
  \begin{displaymath}
    \sum_{i=1}^N\norm{\mu_i}^2_{h,\partial T}\lesssim\norm{\mu_h}^2_{h,\partial T},
  \end{displaymath}
  which can be proved through  standard scaling arguments. Consequently, the desired estimate \eqref{eq:esti decom 2}
  follows immediately from \eqref{eq:s P_iI_h} and \eqref{eq:s mu_i}. 
\end{proof}
By Lemmas \ref{lem:appl X-Z 2}-\ref{lem:esti decom 2} and \eqref{eq:v_h WeWeR_h},
we  immediately obtain the lemma below.
\begin{lem}\label{lem:last 2}
  Under the conditions of Theorem \ref{thm:convergence2}, the relation \eqref{eq:convergence} holds
  with
  \begin{equation}\label{eq:last K 2}
    K\lesssim
    \sup_{\norm{\lam_h}_{\mathcal A_h}=1}\inf_{\mu_h+I_hv_h=\lam_h}
    \sum_{T\in\mathcal T_h}h_T^{-2}\norm{\mu_h}^2_{h,\partial T} +
    \frac
    {
      1+\mathcal N_h
    }
    {
      1-
      \frac
      {
        1+\mathcal N_h
      }
      {
        1-\mathcal N_h
      }
      \bigg[
      (1+\mathcal N_h)\mathcal M_h+\mathcal N_h
      \bigg]^2
    }
    \norm{v_h}^2_{\we{\mathcal A_h}}.
  \end{equation}
\end{lem}
Finally, the rest of the proof of Theorem \ref{thm:convergence2} goes exactly the same way as
that of Theorem \ref{thm:convergence}.

\section{Applications}\label{sec:appl}

This section is devoted to some applications of the algorithm analysis in Section \ref{main-results}  to some existing HDG methods as well as WG methods. 

In the two-level algorithm, {\bf Algorithm
\ref{algo:two-level}}, described in Section \ref{Algorithm}, we set the operator $R_h$ to be the symmetric Gauss-Seidel iteration or one sweep of Gauss-Seidel iteration.  As shown  in {\bf Remark 3.3} and {\bf Appendix A}, the symmetric Gauss-Seidel iteration always satisfies {\bf Assumption \ref{assum:R_h 1}}.  Thus, according to Theorems \ref{thm:convergence}-\ref{thm:convergence2}, we only 
need to verify {\bf Assumptions \ref{assum:a_h}-\ref{assum:M_h N_h}} for the corresponding methods.

We consider the following four types of HDG methods: For any $T\in \mathcal T_h$, $k\geqslant 0$,
\begin{description}
\item[Type 1. ] $V(T) = P_k(T)$, $\bm W(T) = [P_k(T)]^d + P_k(T)\bm x$ and $\alpha_T = 0$. The corresponding 
				HDG scheme \eqref{discretization_hdg} is the   hybridized RT mixed element method 
				(\cite{ArnoldBrezzi1985}).
\item[Type 2. ] $V(T) = P_{k-1}(T)$, $\bm W(T) = [P_k(T)]^d~(k\geqslant 1)$ and $\alpha_T = 0$. The 
					corresponding HDG method is the   hybridized BDM mixed element method 
					(\cite{BrezziDouglasMarini1985}).
\item[Type 3. ] $V(T) = P_k(T)$, $\bm W(T) = [P_k(T)]^d$ and $\alpha_T = O(1)$. The corresponding HDG method was proposed in \cite{Cockburn;unified HDG;2009} and analyzed in \cite{projection_based_hdg}. For the sake of simplicity, we assume  for this HDG method that $\alpha_T$ is
constant on $\partial T$ but it may take different values for different elements $T$.
\item[Type 4. ] $V(T) = P_{k+1}(T)$, $\bm W(T) = [P_k(T)]^d$ and $\alpha_T = O(h_T^{-1})$. The corresponding 
					HDG method was   analyzed in \cite{Li;hdg;2014}
\end{description}

%We recall that Theorem \ref{thm:convergence} requires {\bf Assumptions \ref{assum:a_h}-\ref{assum:R_h 1}, while Theorem \ref{thm:convergence2} \ref{assum:R_h 1} requires \ref{assum:R_h 1}}
%To obtain the convergence estimates in Theorems \ref{thm:convergence}-\ref{thm:convergence2}, we need to verify {\bf Assumptions \ref{assum:a_h}-\ref{assum:R_h 1}}.
For these HDG methods,   {\bf Assumption \ref{assum:a_h}} has been verified in 
\cite{Gopalakrishnan;2003,Cockburn;2014} for {\bf Types 1-3} methods and in \cite{Li;hdg;2014} for {\bf Type 4} method.   Then it suffices to  verify
{\bf Assumption \ref{assum:M_h N_h}}.

For the diffusion-dispersion tensor  $\bm a$, we consider two cases:    piecewise constant coefficients and variable coefficients.

\subsection{Piecewise constant coefficients}\label{ssec:hdg constant}
In this subsection, we assume $\bm a$ to be a piecewise constant matrix, and, without lose of generality, we just take $\bm a$
to be the identity matrix, since   the analysis is the same as that of the former case.

Let $w\in P_1(T)$ and set $\lam=P_T^{\partial}w$ in the local problem \eqref{eq:local hdg}. For {\bf Types 1-2} HDG methods, it is trivial that
\begin{displaymath}
  \bmsigma_{\lam} = \bmnabla w.
\end{displaymath}
For {\bf Type 3} ($k\geqslant 1$) and {\bf Type 4} HDG methods, we can easily obtain
\begin{displaymath}
  P_T^{\partial}u_{\lam}=\lam, \bmsigma_{\lam}=\bmnabla w.
\end{displaymath}
Thus, by the definitions \eqref{De-I_h}-\eqref{De-A_h-We} and  \eqref{we-we-A-h},  for all the mentioned cases above we easily have
\begin{equation}
  \we{\we{\mathcal A_h}} = \we{\mathcal A_h},
\end{equation}
which, together with the definition $\mathcal N_h:=\norm{I-\we{\mathcal A_h}^{-1}\we{\we{\mathcal A_h}}}_{\we{\mathcal A_h}}$ (cf.  \eqref{MN_h}) and {\bf Remark \ref{rem32}}, indicates the following conclusion.
\begin{pro}\label{thm:hdg124}
  For {\bf Types 1-2},  {\bf Type 3} ($k\geqslant 1$) and {\bf Type 4} HDG methods, it holds
  \begin{equation}
\mathcal N_h=0,
\end{equation}
which implies that any choice of $\we{\mathcal R_h}$ satisfying \eqref{M-h-cons} ensures {\bf Assumption \ref{assum:M_h N_h}} to hold.
 % i.e., $\mathcal N_h\lesssim\max\limits_{T\in\mathcal T_h}\alpha_Th_T$.
\end{pro}
For  {\bf Type 3} HDG method in the case $k=0$, we have the following result.
\begin{pro}\label{thm:hdg3}
  For {\bf Type 3} ($k=0$) HDG method, it holds
  \begin{equation}\label{eq:1}
   % \norm{I-\we{\mathcal A_h}^{-1}\we{\we{\mathcal A_h}}}_{\we{\mathcal A_h}}
  \mathcal N_h  \lesssim h, %\max_{T\in\mathcal T_h}\alpha_Th_T,
 \end{equation}
 which implies that  sufficiently small mesh size $h$ can ensure {\bf Assumption \ref{assum:M_h N_h}} to hold  if   $\we{\mathcal R_h}$ satisfies \eqref{M-h-cons}, i.e. $\mathcal M_h<1$, with $\mathcal M_h:=\norm{I-\we{\mathcal R_h}\we{\mathcal A_h}}_{\we{\mathcal A_h}}$ being  independent of $h$.
 % i.e., $\mathcal N_h\lesssim\max\limits_{T\in\mathcal T_h}\alpha_Th_T$.
\end{pro}
\begin{proof}
  For any $w\in P_1(T)$, set $\lam=P_T^{\partial}w$ in the local problem \eqref{eq:local hdg}, then
  it holds
  \begin{equation}\label{eq:k=0}
    \bmsigma_{\lam}=\bmnabla w,~u_{\lam}=m_T(w)= \frac{1}{|\partial T|}\int_{\partial T}w.
  \end{equation}
  Consider an auxiliary  problem as follows: for any $u_h\in V_h$, seek $v_h\in V_h$ such that
  \begin{displaymath}
    (u_h,w_h)_{\we{\we{\mathcal A_h}}}=(v_h,w_h)_{\we{\mathcal A_h}},
    ~\forall w_h\in V_h.
  \end{displaymath}
  By \eqref{eq:k=0}, we easily obtain
  \begin{displaymath}
    (\bmnabla u_h,\bmnabla w_h)+\bndint{\alpha_T(m_T(u_h)-P_T^{\partial}u_h),
      m_T(w_h)-P_T^{\partial}w_h} =
    (\bmnabla v_h,\bmnabla w_h), ~\forall w_h\in V_h,
  \end{displaymath}
  and it follows 
  \begin{equation}\label{eq:err}
    \begin{split}
      (\bmnabla(u_h-v_h),\bmnabla w_h)=-\bndint{\alpha_T(m_T(u_h)-P_T^{\partial}u_h),
        m_T(w_h)-P_T^{\partial}w_h},
      ~\forall w_h\in V_h.
    \end{split}
  \end{equation}
  Since 
  \begin{equation}\label{eq:I-m_T}
    \norm{v-m_T(v)}_{\partial T} \lesssim h_T^{\frac{1}{2}}\norm{\bmnabla v}_T,
    ~\forall v\in H^1(T),
  \end{equation}
   taking $w_h=u_h-v_h$ in \eqref{eq:err} we have
  \begin{displaymath}
    \begin{array}{rcll}
      \norm{u_h-v_h}^2_{\we{\mathcal A_h}}
      &\leqslant& \sum\limits_{T\in\mathcal T_h}\alpha_T\norm{m_T(u_h)-u_h}_{\partial T}
                  \norm{m_T(w_h)-w_h}_{\partial T}&\\
      &\lesssim& \sum\limits_{T\in\mathcal T_h}\alpha_Th_T\norm{\bmnabla u_h}_T
                 \norm{\bmnabla w_h}_T &\text{(by \eqref{eq:I-m_T})}\\
      &\lesssim& \max\limits_{T\in\mathcal T_h}\alpha_Th_T\norm{u_h}_{\we{\mathcal A_h}}
                 \norm{u_h-v_h}_{\we{\mathcal A_h}},&
    \end{array}
  \end{displaymath}
  which leads to
  \begin{equation}\label{eq:567}
    \norm{u_h-v_h}_{\we{\mathcal A_h}}\lesssim
    \max\limits_{T\in\mathcal T_h}\alpha_Th_T\norm{u_h}_{\we{\mathcal A_h}},
  \end{equation}
  i.e.
  \begin{displaymath}
    \norm{(I-\we{\mathcal A_h}^{-1}\we{\we{\mathcal A_h}})u_h}_{\we{\mathcal A_h}}
    \lesssim\max_{T\in\mathcal T_h}\alpha_Th_T\norm{u_h}_{\we{\mathcal A_h}},
    ~\forall u_h\in V_h,
  \end{displaymath}
  which, by recalling $\alpha_T=  O(1)$, yields \eqref{eq:1} immediately.
\end{proof}
\begin{rem}
  By Theorems \ref{thm:convergence}-\ref{thm:convergence2}, it is easy to derive the convergence rate (independent of mesh size) of a V-cycle HDG multigrid in
  \cite{Gopalakrishnan;2009}, where full elliptic regularity ($\Omega$
  was assumed to be convex) was required. However, our analysis does not   require  full  regularity.
\end{rem}
\begin{rem}
  From Theorem \ref{thm:convergence} and Proposition \ref{thm:hdg3}, in order to the convergence of {\bf Algorithm
\ref{algo:two-level}}, we have to require 
  $h$ to be small
  enough. This
  is in agreement with the theoretical result   in \cite{Cockburn;2014}.
\end{rem}

Suppose $\we{\mathcal R_h}$ satisfies \eqref{M-h-cons}. We summarize   this subsection as follows:  
\begin{itemize}
\item For {\bf Type 1-2}, {\bf Type 3} ($k\geqslant 1$) and {\bf Type 4} HDG
  methods, both Theorem \ref{thm:convergence} and Theorem \ref{thm:convergence2}
  hold with
  $$
  K\lesssim 1+\frac{1}{1-\norm{I-\overline{\we{\mathcal R_h}}\we{\mathcal A_h}}_{\we{\mathcal A_h}}}.
  $$
  
\item For {\bf Type 3} ($k=0$) HDG method,  the mesh size $h$ should be sufficiently small   to ensure the convergence of {\bf Algorithm
\ref{algo:two-level}}.
\end{itemize}

\subsection{Variable coefficients}\label{Variable coefficients}
In this subsection, we assume $\bm a\in [W^{1,\infty}(\mathcal T_h)]^{d\times d}$, where
$W^{1,\infty}(\mathcal T_h):=\{a\in L^{\infty}(\Omega):\bmnabla a|_T\in [L^{\infty}(T)]^d,
~\forall T\in\mathcal T_h\}$. 
In the analysis below, we only consider {\bf Types 1-2} and
{\bf Type 3} ($k\geqslant 1$) HDG methods, since by the technique used here, it is easy to derive similar
results for other HDG methods. Following the same routines as in Section \ref{ssec:hdg constant} (cf. Propositions \ref{thm:hdg124}-\ref{thm:hdg3}), we only need to  estimate the number 
 $\mathcal N_h=\norm{I-\we{\mathcal A_h}^{-1}\we{\we{\mathcal A_h}}}_{\we{\mathcal A_h}}$. 
%{\bf Assumption \ref{assum:M_h N_h}} holds with $\mathcal N_h\lesssim h$, i.e.,
\par
\begin{lem}
  For {\bf Types 1-2} HDG methods, it holds 
\begin{equation}\label{eq:goal}
  \mathcal N_h \lesssim h.
\end{equation}

\end{lem}
\begin{proof}
  For any $w\in P_1(T)$, set $\lam=P_T^{\partial}w$  in the local problem \eqref{eq:local hdg}. Then it is easy to show
  \begin{equation}
    \norm{\bmsigma_{\lam}}_T\lesssim\vertiii{\lam}_{h,\partial T}\sim\norm{\bmnabla w}_T.
    \label{eq:813}
  \end{equation}
On the other hand,  by  \eqref{eq:local hdg} we also have
  \begin{displaymath}
    (\bm c\bmsigma_{\lam}-\bmnabla w,\bmsigma_{\lam}-\bar{\bm c}^{-1}\bmnabla w)_T = 0
  \end{displaymath}
  with $\bar{\bm c}:=\frac{1}{|T|}\int_T\bm c$. Thus  it holds 
  \begin{displaymath}
    \begin{split}
      (\bar{\bm c}\bmsigma_{\lam}-\bmnabla w,\bmsigma_{\lam}-\bar{\bm c}^{-1}\bmnabla w)_T
      &=((\bar{\bm c}-\bm c)\bmsigma_{\lam},\bmsigma_{\lam}-\bar{\bm c}^{-1}\bmnabla w)_T
      + (\bm c\bmsigma_{\lam}-\bmnabla w,\bmsigma_{\lam}-\bar{\bm c}^{-1}\bmnabla w)_T\\
      &=((\bar{\bm c}-\bm c)\bmsigma_{\lam},\bmsigma_{\lam}-\bar{\bm c}^{-1}\bmnabla w)_T\\
      &\lesssim h_T\norm{\bmsigma_{\lam}}_T\norm{\bmsigma_{\lam}-\bar{\bm c}^{-1}\bmnabla w}_T,
    \end{split}
  \end{displaymath}
  where in the last "$\lesssim$" we have used the standard  estimate
\begin{displaymath}
  \norm{\bar{\bm c}-\bm c}_{L^{\infty}(T)}\lesssim h_T\norm{\bmnabla \bm c}_{L^{\infty}(T)}.
\end{displaymath}
Hence it follows
  \begin{equation}
    \norm{\bar{\bm c}\bmsigma_{\lam}-\bmnabla w}_T \lesssim h_T\norm{\bmsigma_{\lam}}_T,
  \end{equation}
 which implies
    \begin{equation}
    \begin{split}
      \norm{\bm c\bmsigma_{\lam}-\bmnabla w}_T
      &\leqslant\norm{(\bm c-\bar{\bm c})\bmsigma_{\lam}}_T + \norm{\bar{\bm c}\bmsigma_{\lam}-\bmnabla w}_T\\
      &\lesssim h_T\norm{\bmsigma_{\lam}}_T.
    \end{split}
  \end{equation}
  This estimate, together with   $\eqref{eq:813}$, yields
  \begin{equation}\label{eq:124}
    \norm{\bm c\bmsigma_{\lam}-\bmnabla w}_T
    \lesssim h_T\norm{\bmnabla w}_T.
  \end{equation}
  Finally, for any $v_h\in V_h,$   taking $w=v_h|_T$ in \eqref{eq:813} and \eqref{eq:124} with $\lambda=I_hv_h|_T$, from the definitions \eqref{De-I_h}-\eqref{De-A_h-We} and  \eqref{we-we-A-h}, it follows 
    \begin{displaymath}
    \begin{split}
      \left|((\we{\we{\mathcal A_h}}-\we{\mathcal A_h})v_h,v_h)\right|
      &=\left|\norm{\bmsigma_{I_hv_h}}^2_{\bm c}-\norm{\bm a\bmnabla v_h}^2_{\bm c}\right|\\
      &\leqslant(\norm{\bmsigma_{I_hv_h}}_{\bm c}+\norm{\bm a\bmnabla v_h}_{\bm c})
      \norm{\bmsigma_{I_hv_h}-\bm a\bmnabla v_h}_{\bm c}\\
      &\lesssim h\norm{v_h}^2_{\we{\mathcal A_h}},
          \end{split}
  \end{displaymath}
  which gives the desired estimate  \eqref{eq:goal}.
\end{proof}
\begin{rem}
  For {\bf Type 1} HDG method ($k=0$), if we  redefine $\we{\mathcal A_h}$ as
    \begin{displaymath}
    (\we{\mathcal A_h}u_h,v_h)=(\bar{\bm c}^{-1}\bmnabla u_h,\bmnabla v_h),~\forall u_h,v_h\in V_h,
  \end{displaymath}
  then it holds $\we{\mathcal A_h}=\we{\we{\mathcal A_h}}$ and $ \mathcal N_h =0$. This is a trivial modification
  of \cite{Chen;1994;a}.
\end{rem}
\par
Next we consider {\bf Type 3} HDG method.
\begin{thm}
  For {\bf Type 3} HDG method ($k\geqslant 1$), the estimate \eqref{eq:goal} holds.
\end{thm}
\begin{proof}
  Let $w\in P_1(T)$ and set $\lam=P_T^{\partial}w$   in the local problem \eqref{eq:local hdg}. It is easy to obtain
  \begin{subequations}
    \begin{eqnarray}
      (\bm c\bmsigma_{\lam}-\bmnabla w,\bm\tau)_T+(u_{\lam}-w,div\bm\tau)_T &=& 0,
                                                                                ~\forall\bm\tau\in\bm W(T),\label{986}\\
      -(v,div\bmsigma_{\lam})_T+\bint{\alpha_T(u_{\lam}-\lam),v}{\partial T} &=& 0,
                                                                                 ~\forall v\in V(T).\label{987}
    \end{eqnarray}
  \end{subequations}
  Taking $\bm\tau=\bmsigma_{\lam}-\bar{\bm c}^{-1}\bmnabla w$, $v=u_{\lam}-w$
  and adding
  \eqref{986} and \eqref{987}, we have
  \begin{equation}
    (\bm c\bmsigma_{\lam}-\bmnabla w,\bmsigma_{\lam}-\bar{\bm c}^{-1}\bmnabla w)_T+
    \bint{\alpha_T(u_{\lam}-\lam),u_{\lam}-\lam}{\partial T} = 0.
  \end{equation}
 This relation yields
  \begin{displaymath}
    \begin{split}
      &(\bar{\bm c}\bmsigma_{\lam}-\bmnabla w,\bmsigma_{\lam}-\bar{\bm c}^{-1}\bmnabla w)_T
      +\bint{\alpha_T(u_{\lam}-\lam),u_{\lam}-\lam}{\partial T}\\
      =&((\bar{\bm c}-\bm c)\bmsigma_{\lam},\bmsigma_{\lam}-\bar{\bm c}^{-1}\bmnabla w)_T\\
      \lesssim&h_T\norm{\bmsigma_{\lam}}_T\norm{\bmsigma_{\lam}-\bar{\bm c}^{-1}\bmnabla w}_T,
    \end{split}
  \end{displaymath}
  which implies
  \begin{displaymath}
    \norm{\bar{\bm c}\bmsigma_{\lam}-\bmnabla w}_T+\alpha_T^{\frac{1}{2}}\norm{u_{\lam}-\lam}_{\partial T}
    \lesssim h_T\norm{\bmsigma_{\lam}}_T.
  \end{displaymath}
  Hence it follows
    \begin{equation}\label{eq:543}
    \norm{\bm c\bmsigma_{\lam}-\bmnabla w}_T+\alpha_T^{\frac{1}{2}}\norm{u_{\lam}-\lam}_{\partial T}
    \lesssim h_T\norm{\bmsigma_{\lam}}_T,
  \end{equation}
  % Since it is trivial to verify that
  % \begin{equation}\label{eq:888888}
  %   \norm{\bmsigma_{\lam}}_T\lesssim \vertiii{\lam}_{h,\partial T}\sim\norm{\bmnabla w}_T,
  % \end{equation}
  which, together with $\eqref{eq:813}$, shows
  \begin{equation}\label{eq:999999}
    \norm{\bm c\bmsigma_{\lam}-\bmnabla w}_T+\alpha_T^{\frac{1}{2}}\norm{u_{\lam}-\lam}_{\partial T}
    \lesssim h_T\norm{\bmnabla w}_T.
  \end{equation}
   Finally, for any $v_h\in V_h,$   taking $w=v_h|_T$ in \eqref{eq:813} and \eqref{eq:999999} with $\lambda=I_hv_h|_T$, from the definitions \eqref{De-I_h}-\eqref{De-A_h-We} and  \eqref{we-we-A-h}, it follows 
  \begin{displaymath}
    \begin{split}
      \bigg|((\we{\we{\mathcal A_h}}-\we{\mathcal A_h})v_h,v_h)\bigg|
      &=\bigg|\norm{\bmsigma_{I_hv_h}}^2_{\bm c}+\sum_{T\in\mathcal T_h}\alpha_T
      \norm{u_{I_hv_h}-I_hv_h}^2_{\partial T}-\norm{\bm a\bmnabla v_h}^2_{\bm c}\bigg|\\
      &\lesssim h\norm{v_h}^2_{\we{\mathcal A_h}},
    \end{split}
  \end{displaymath}
  which implies \eqref{eq:goal}. 
\end{proof}
\begin{rem}
  We note  that our analysis only requires the regularity estimate
 \eqref{eq:reg} with $\alpha \in [0,1]$, while  the analysis   in \cite{Cockburn;2014} requires
    $\alpha \in (0.5,1]$.
\end{rem}

Similar to Section \ref{ssec:hdg constant}, we summarize   this subsection as follows: 
\begin{itemize}
\item When the tensor
$\bm a$ is  not piecewise constant but piecewise smooth, the convergence of the two-level algorithm, {\bf Algorithm
\ref{algo:two-level}}, for the HDG methods can still be obtained, as long as the  mesh size $h$ is   small enough.
\end{itemize}

% Thus the two-level algorithm is not very robust for the case that
% $\bm a$ is variable. However, we would remark that the two-level algorithm can stiff provide
% robust preconditioners: For {\bf Type 1-2}, {\bf Type 3} ($k\geqslant 1$) and {\bf Type 4}
% HDG methods, let $\mathcal B_h$ be defined in the case that $\bm a$ is the identity
% matrix, then it is easy to prove that $\mathcal B_h$ is a robust preconditioner for any SPD $\bm a$.

\subsection{Application to weak Galerkin methods}\label{sec:wg append}
In this subsection, we shall show our analysis can also be extended to the WG methods. Unless otherwise specified, we adopt the notations introduced in
section \ref{sec:notation}. 
%We first introduce two spaces as follows:
%\begin{eqnarray}
%  V_h&:=&\{v_h\in L^2(\Omega):v_h|_T\in V(T),~\forall T\in\mathcal T_h\};\\
%  \bm W_h&:=&\{\bm\tau_h\in [L^2(\Omega)]^d:\bm\tau_h|_T\in\bm W(T),~\forall T\in\mathcal T_h\}.
%\end{eqnarray}

Following \cite{Wang;2013}, we introduce the weak gradient operators as follows.
For any $T\in \mathcal T^h$, define $\bmnabla_w^i:L^2(T)\to\bm W(T)$ by
\begin{equation}
  (\bmnabla_w^iv,\bm q)_T = -(v,div\bm q)_T,~\forall\bm q \in\bm W(T),\forall v\in L^2(T),
\end{equation}
and $\bmnabla_w^b:L^2(\partial T)\to\bm W(T)$ by
\begin{equation}
  (\bmnabla_w^b\lam,\bm q)_T = \bint{\lam,\bm q\cdot\bm n}{\partial T},
  \forall\bm q\in\bm W(T),\forall\lam\in L^2(\partial T),
\end{equation}
where $\bm n$ denotes the unit outward normal vector to $\partial T$.

The WG framework for the model problem \eqref{eq:model} reads as follows(\cite{Wang;2013}):
seek $(u_h,\lam_h)\in V_h\times M_h$ such that
\begin{equation}\label{eq:wg appendix}
  \sum_{T\in\mathcal T_h}(\bm a(\bmnabla_w^iu_h+\bmnabla_w^b\lam_h),\bmnabla_w^iv_h+\bmnabla_w^b\mu_h)_T
  +s_h((u_h,\lam_h),(v_h,\mu_h))
  =(f,v_h), \ \forall (v_h,\mu_h)\in V_h\times M_h,
\end{equation}
% holds for all $(v_h,\mu_h)\in V_h\times M_h$, where
where
\begin{equation}
  s_h((u_h,\lam_h,(v_h,\mu_h)):=
  \bndint{\alpha_T(P_T^{\partial}u_h-\lam_h),P_T^{\partial}v_h-\mu_h}.
\end{equation}
Denote $\bm\sigma_h:=\bmnabla_w^iu_h+\bmnabla_w^b\lam_h$, then the WG model \eqref{eq:wg appendix} is
equivalent to the following HDG-like scheme: seek $(u_h,\lam_h,\bm\sigma_h)\in V_h\times M_h\times\bm W_h$ such that
\begin{subequations}\label{eq:wg}
  \begin{eqnarray}
    (\bm\sigma_h,\bm\tau_h)+(u_h,div\bm\tau_h)-\bndint{\lam_h,\bm\tau_h\cdot\bm n}
    &=& 0, \ \forall \bm\tau_h\in \bm W_h,\\
    -(v_h,div_h(P_h^{\bm W}(\bm a\bm\sigma_h)))+\bndint{\alpha_T(P_T^{\partial}(u_h-\lam_h),v_h)}
    &=& (f,v_h),\ \forall v_h\in V_h,\\
    \bndint{P_h^{\bm W}(\bm a\bm\sigma_h)\cdot\bm n-\alpha_T(P_T^{\partial}u_h-\lam_h),\mu_h}
    &=& 0,\ \forall \mu_h\in M_h,
  \end{eqnarray}
\end{subequations}
% hold for all $(v_h,\mu_h,\bm\tau_h)\in V_h\times M_h\times\bm W_h$,
where $P_h^{\bm W}:[L^2(\Omega)]^d
\to\bm W_h$ denotes the standard $L^2$-orthogonal projection operator.

We define the local problem as follows: for any $\lam\in L^2(\partial T)$, seek
$(u^{wg}_{\lam},\bmsigma^{wg}_{\lam})\in V(T)\times\bm W(T)$ such that
\begin{subequations}\label{eq:local wg}
  \begin{eqnarray}
    (\bmsigma^{wg}_{\lam},\bm\tau)_T+(u^{wg}_{\lam},div\bm\tau)_T
    &=& \bint{\lam,\bm\tau\cdot\bm n}{\partial T},\ \forall \bm\tau\in \bm W(T),\\
    -(v,div(P_T^{\bm W}(\bm a\bmsigma^{wg}_{\lam}))_T +
    \bint{\alpha_TP_T^{\partial}u^{wg}_{\lam},v}{\partial T}
    &=&\bint{\alpha_T\lam,v}{\partial T},\ \forall v\in V(T),
  \end{eqnarray}
\end{subequations}
% hold for all $(v,\bm\tau)\in V(T)\times\bm W(T)$,
where $P_T^{\bm W}:[L^2(T)]^d\to\bm W(T)$ denotes
the local $L^2$-orthogonal projection operator.

Similar to   Theorem 2.1 in \cite{Cockburn;unified HDG;2009},   the following proposition holds.
\begin{pro}
  Suppose $(u_h,\lam_h)\in V_h\times M_h$ solves the WG model \eqref{eq:wg appendix}, then $\lam_h$ can
  be obtained by solving the system
  \begin{equation}\label{eq:global wg}
    a_h^{wg}(\lam_h,\mu_h) = (f,u^{wg}_{\mu_h}),~\forall\mu_h\in M_h,
  \end{equation}
  where
  \begin{equation}
    a_h^{wg}(\lam_h,\mu_h)=(\bm a\bmsigma^{wg}_{\lam_h},\bmsigma^{wg}_{\mu_h})+
    \bndint{\alpha_T(P_T^{\partial}u^{wg}_{\lam_h}-\lam_h),P_T^{\partial}u^{wg}_{\mu_h}-\mu_h}.
  \end{equation}
\end{pro}
\begin{rem}
  Similar to the HDG methods, once $\lam_h$ is resolved, $u_h$ and $\bmsigma_h$ in \eqref{eq:wg} can be obtained in
  an element-by-element fashion.
  \end{rem}
  
  %\begin{rem}\label{rem46} 
 When applying the two-level algorithm, {\bf Algorithm
\ref{algo:two-level}}, to   WG methods based the   model  \eqref{eq:global wg}, %\eqref{eq:wg appendix} or its equivalent formulation \eqref{eq:wg},  
 we set the operator $R_h$ to be the symmetric Gauss-Seidel iteration or one sweep of Gauss-Seidel iteration. Similar to the HDG methods, one can easily show that  the symmetric Gauss-Seidel iteration always satisfies {\bf Assumption \ref{assum:R_h 1}}. 

When  $\bm a$ is a piecewise constant matrix,  from the HDG-like formulation \eqref{eq:wg} we can see that  
  the    WG framework \eqref{eq:wg appendix}  is essentially equivalent to
  the corresponding HDG framework.  As a result,   the convergence of the algorithm for he WG methods is as same as that for the corresponding HDG methods.

For more general case of $\bm a$, 
by using the same technique  as in \cite{Gopalakrishnan;2003,Cockburn;2014,Li;hdg;2014}
it is easy to verify that
\begin{equation}
  a_h(\lam_h,\lam_h)\sim a^{wg}_h(\lam_h,\lam_h),~\forall\lam_h \in M_h.
\end{equation}
%Thus, if {\bf Assumptions \ref{assum:a_h}-\ref{assum:R_h 1}} are true, through similar analysis, we can prove that {\bf Algorithm
%\ref{algo:two-level}} is also convergent for solving \eqref{eq:global wg}.
% then $\mathcal B_h$ can be taken as an good preconditioner for the corresponding WG methods.
Then {\bf Assumptions \ref{assum:a_h}} is obviously true for the WG methods.  Following the same  routines as in Section \ref{Variable coefficients}, one can derive
 the estimate \eqref{eq:goal}. Therefore, similar convergence results of {\bf Algorithm
\ref{algo:two-level}} for HDG methods  also hold for the WG methods.

\section{Numerical results}\label{sec:numerical}
In this section, we provide some numerical experiments in 2-dimensional case to support our theoretical analysis.
We only consider {\bf Type 3} HDG method with $\alpha_T =1,\forall T\in\mathcal T_h$. For more numerical results we refer
  to \cite{Cockburn;2014}.

In the first experiment, we set $\Omega = (-1, 1)\times (0, 1)\bigcup (0, 1)\times (-1, 0]$ and define
$\bm a(x, y) = \text{ diag } (a(x,y), a(x, y))$ with
\begin{displaymath}
  a(x,y) := \left\{
    \begin{array}{rrr}
      1, & -1<x<0,&0<y<1;\\
      5, & 0<x<1,&0<y<1;\\
      10,& 0<x<1,&-1<y<0.
    \end{array}
  \right.
\end{displaymath}
Given an initial triangulation $\mathcal T_0$ of $\Omega$, we produce a sequence of triangulations
$\{\mathcal T_j:j=1,2,\cdots,5\}$ by a simple procedure: $\mathcal T_{j+1}$ is obtained by connecting
the midpoints of each face of $\mathcal T_j$ for $j=0,1,\cdots,4$. $\mathcal T_0$ and $\mathcal T_1$
are presented in Figure \ref{Lmesh0_Lmesh1} for clarity. For each $\mathcal T_j$ ($j=1,2,\cdots,5$),
we set $\mathcal T_h=\mathcal T_j$ and construct $\we{\mathcal R_h}$   by using the standard $\cal V$-cycle
multigrid method based on the triangulations $\{\mathcal T_i:i=0,1,\cdots, j\}$, i.e.
$I-\overline{\we{\mathcal R_h}}\we{\mathcal A_h}$ denotes the error transfer operator of one $\cal V$-cycle
iteration. Here we set $\mathcal R_h$ and all smoothers encountered in the construction of
$\we{\mathcal R_h}$ to be the symmetric Gauss-Seidel method with $m_0$ and $m_1$ iterations respectively.
Using the standard nodal basis for $M_h$, we let $A_h$ be the stiffness matrix arising from the bilinear
form \eqref{eq:a_h hdg}. Suppose we are to solve $A_hx = b_h$ where $b_h$ is a zero vector, and we
take $x_0=(1,1,\cdots, 1)^t$ to be the initial value, rather than the zero vector   presented in
Algorithm \ref{algo:two-level}. We stop Algorithm \ref{algo:two-level} until the initial error,
i.e. $\sqrt{x_0^tA_hx_0}$, is reduced by a factor of $10^{-8}$.  The corresponding numerical results
(the number of iterations in  Algorithm \ref{algo:two-level}) are presented in Table \ref{tab:first experiment}.

The second experiment is a simple modification of the first one: we set $\mathcal R_h$ to be one
sweep of Gauss-Seidel iteration. The corresponding numerical results are presented in Table
\ref{tab:2nd experiment}.
\begin{figure}[H]
  \begin{center}
    \includegraphics[width=0.4\linewidth]{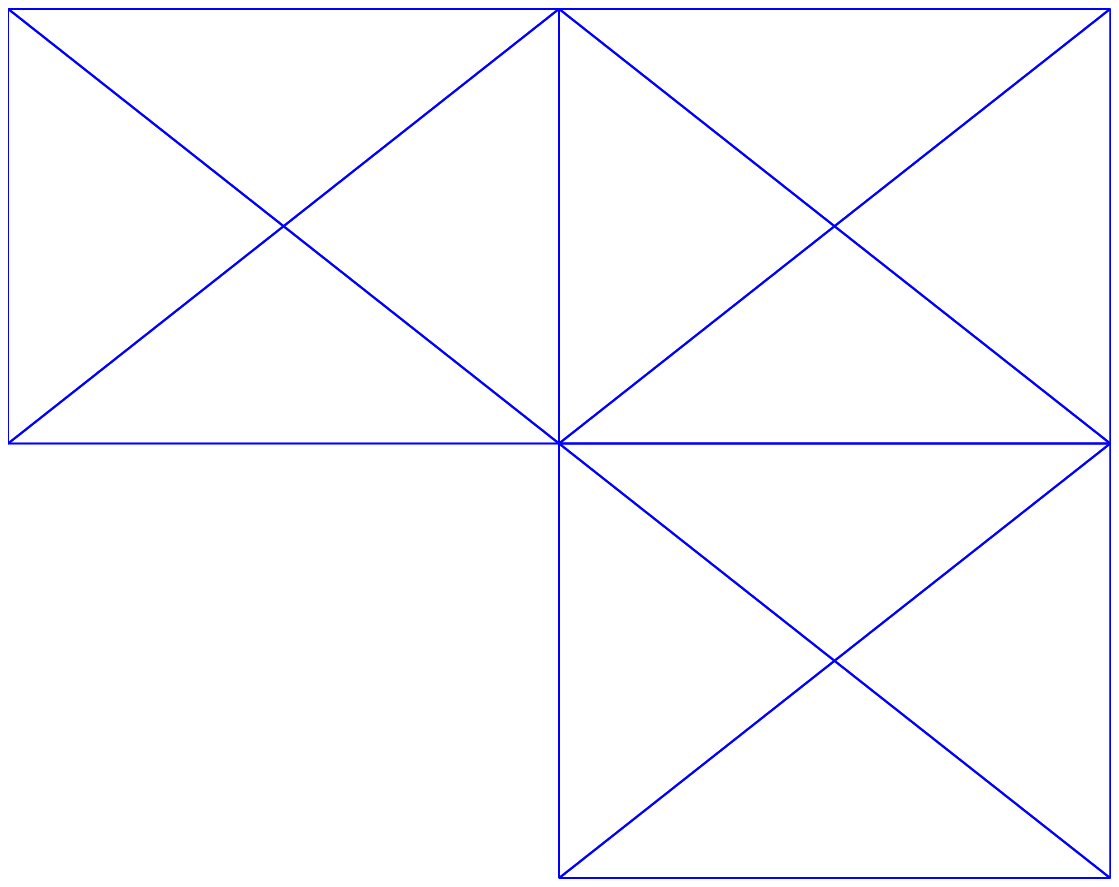}
    \includegraphics[width=0.4\linewidth]{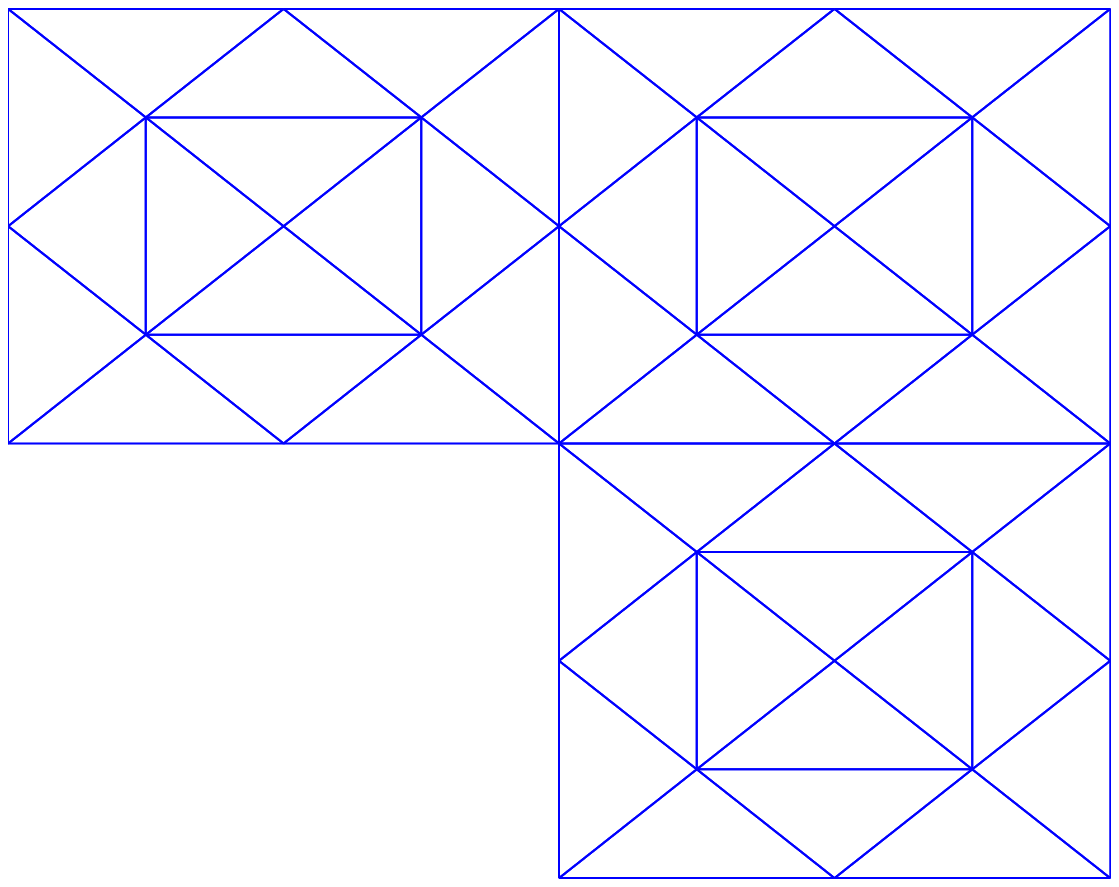}
  \end{center}
  \caption{$\mathcal T_0$ (left) and $\mathcal T_1$ (right)}\label{Lmesh0_Lmesh1}
\end{figure}
\par
\begin{table}[H]
  \ttabbox
  {
    \begin{tabular}{p{3.pt} p{3.pt} p{3.pt} p{3.pt} p{3.pt} p{3.pt} p{3.pt}}
      \toprule
      $k$ & $m_0$ & $\mathcal T_1$ & $\mathcal T_2$ & $\mathcal T_3$ & $\mathcal T_4$ & $\mathcal T_5$\\
      \midrule
      0 & 1 & 19 & 18 & 19 & 19 & 19 \\
          &2& 13 & 13 & 14 & 14& 15\\
          &3& 10 & 12 & 13 & 13 & 14\\
      \midrule
      1&1& 20 & 21 & 21 & 20 & 20\\
          &2& 13 & 14 & 14 & 15 & 15\\
          &3& 11 & 12 & 13 & 13 & 14\\
      \bottomrule
      \text{$\qquad\qquad m_1=1$}
    \end{tabular}
    \begin{tabular}{p{3.pt} p{3.pt} p{3.pt} p{3.pt} p{3.pt} p{3.pt} p{3.pt}}
      \toprule
      $k$ & $m_0$ & $\mathcal T_1$ & $\mathcal T_2$ & $\mathcal T_3$ & $\mathcal T_4$ & $\mathcal T_5$\\
      \midrule
      0 & 1 & 17 & 18 & 17 & 17 & 17 \\
          &2& 12 & 12 & 12 & 12& 12\\
          &3& 10 & 10 & 10 & 11 & 11\\
      \midrule
      1&1& 20 & 20 & 20 & 20 & 19\\
          &2& 12 & 13 & 13 & 13 & 12\\
          &3& 10 & 10 & 11 & 11 & 11\\
      \bottomrule
      \text{$\qquad\qquad m_1=2$}
    \end{tabular}
    \begin{tabular}{p{3.pt} p{3.pt} p{3.pt} p{3.pt} p{3.pt} p{3.pt} p{3.pt}}
      \toprule
      $k$ & $m_0$ & $\mathcal T_1$ & $\mathcal T_2$ & $\mathcal T_3$ & $\mathcal T_4$ & $\mathcal T_5$\\
      \midrule
      0 & 1 & 17 & 17 & 17 & 17 & 17 \\
          &2& 12 & 11 & 11 & 11& 11\\
          &3& 10 & 9 & 10 & 10 & 10\\
      \midrule
      1&1& 20 & 20 & 20 & 19 & 19\\
          &2& 12 & 13 & 12 & 12 & 12\\
          &3& 10 & 10 & 10 & 10 & 10\\
      \bottomrule
      \text{$\qquad\qquad m_1=3$}
    \end{tabular}
  }
  {
    \caption{Numerical results for the first experiment}\label{tab:first experiment}
  }
\end{table}
\begin{table}[H]
  \ttabbox
  {
    \begin{tabular}{p{10.pt} p{10.pt} p{10.pt} p{10.pt} p{10.pt} p{10.pt} p{10.pt}}
      \toprule
      $k$ & $m_1$ & $\mathcal T_1$ & $\mathcal T_2$ & $\mathcal T_3$ & $\mathcal T_4$ & $\mathcal T_5$\\
      \midrule
      0 & 1 & 22 &  24 & 24 & 23 & 23 \\
          & 2& 22 & 23 & 23 & 23 & 22\\
          &3& 21 & 23 & 23 & 22 & 22\\
      \midrule
      1&1& 34 & 34 & 34 & 34 & 34\\
          &2& 34 & 34 & 34 & 34 & 34\\
          &3& 34 & 34 & 34 & 34 & 34\\
      \bottomrule
    \end{tabular}
  }
  {
    \caption{Numerical results for the second experiment}\label{tab:2nd experiment}
  }
\end{table}
In the third experiment, we set $\Omega = (-1,1)\times(-1,1)$ and define $\bm a(x,y)=\text{diag}(a(x,y), a(x,y))$
with
\begin{displaymath}
  a(x,y):=\left\{
    \begin{array}{rrr}
      1, & -1 < x < 0, & -1 < y < 0;\\
      7, &  0 < x < 1, & -1 < y < 0;\\
      17,&  0 < x < 1, &  0 < y < 1;\\
      3, & -1 < x < 0, & 0 < y < 1.
    \end{array}
  \right.
\end{displaymath}
We show the first two triangulations $\mathcal T_0$ and $\mathcal T_1$ in Figure \ref{mesh0_mesh1}
and produce a sequence of triangulations $\{\mathcal T_j:j=0,1,\cdots,25\}$ in a successive way:
$\mathcal T_{j+1}$ ($j=2,3,\cdots,24$) is obtained by refining the smallest square containing the
origin in $\mathcal T_j$ (in $\mathcal T_1$, the vertexes of the square to refine is in red color)
as same as  what has been done from $\mathcal T_0$ to $\mathcal T_1$.  $\mathcal T_{25}$ is shown in Figure
\ref{mesh2}. The difference of the two-level algorithm between this experiment and the first one is
that we simply take $\we{\mathcal R_h}=\we{\mathcal A_h}^{-1}$ here. The corresponding numerical results
are presented in Table \ref{tab:third experiment}.
\begin{figure}[H]
  \begin{center}
    \includegraphics[width=0.4\linewidth]{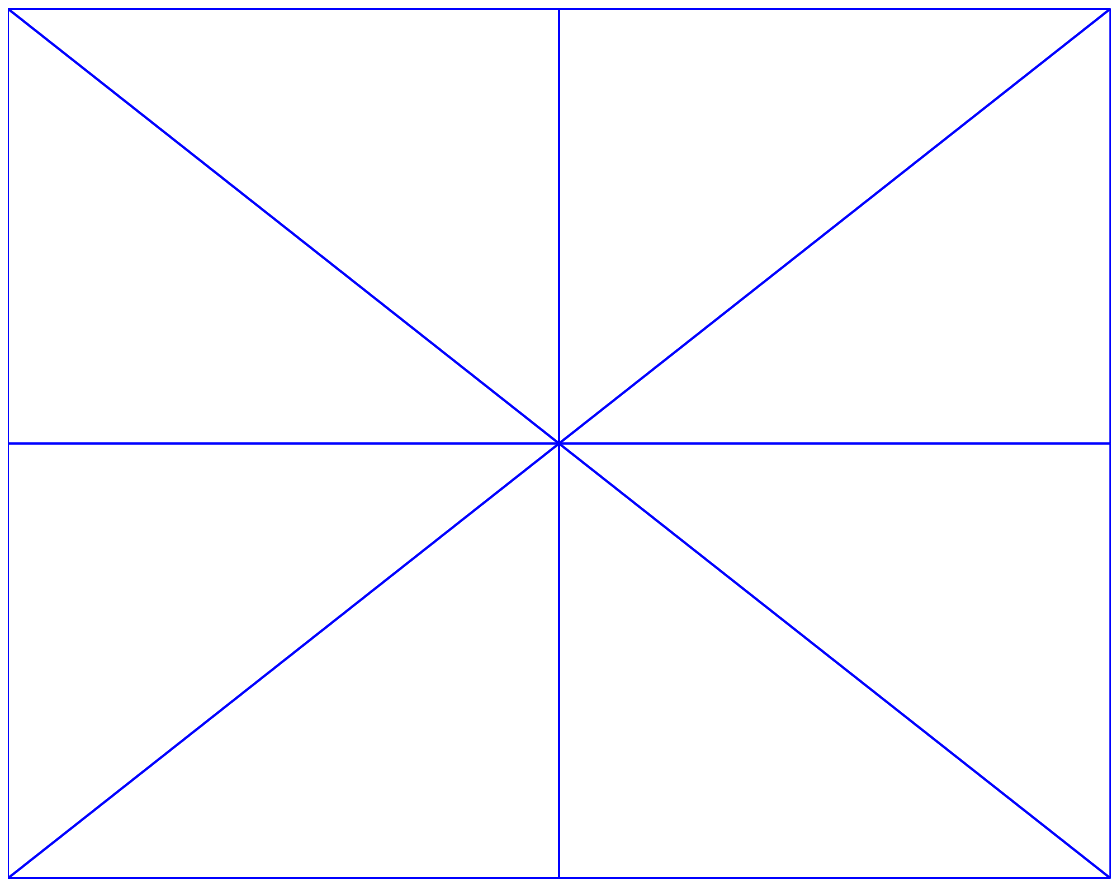}
    \includegraphics[width=0.4\linewidth]{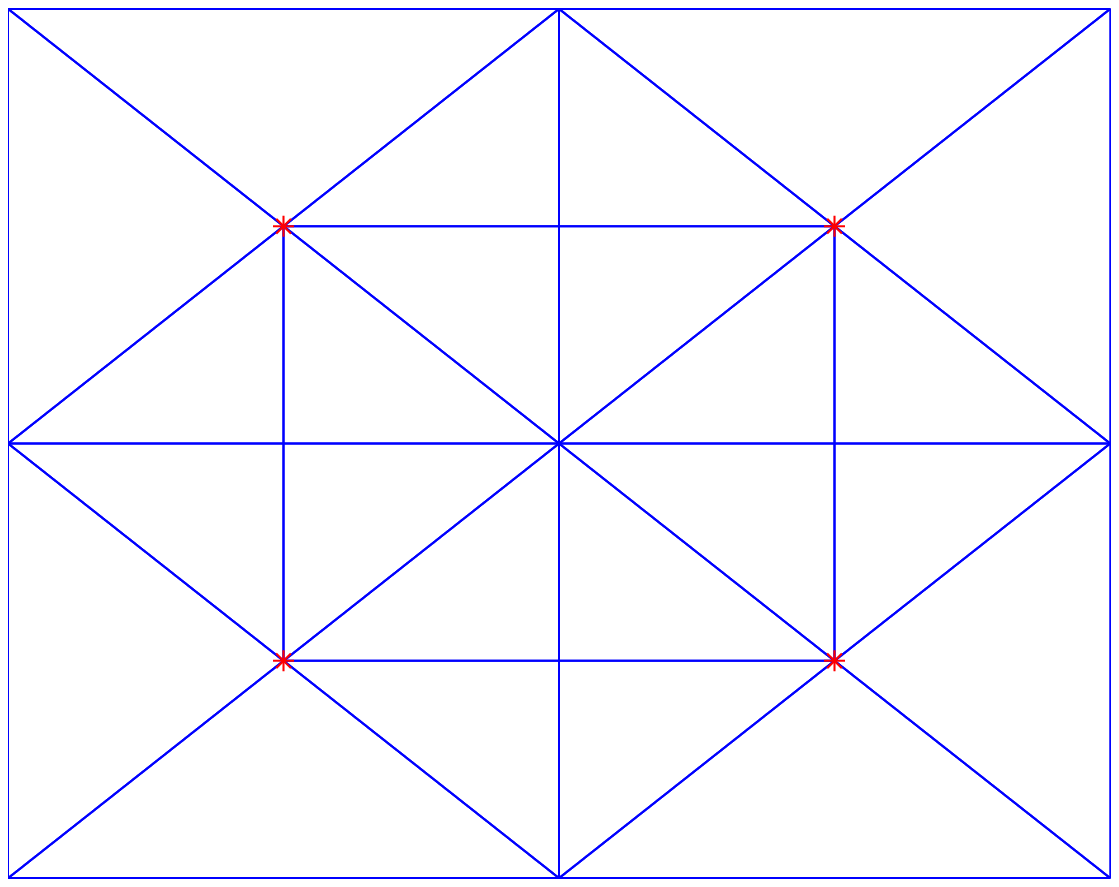}
  \end{center}
  \caption{$\mathcal T_0$ (left) and $\mathcal T_1$ (right)}\label{mesh0_mesh1}
\end{figure}
\par
\begin{figure}[H]
  \begin{center}
    \includegraphics[width=0.4\linewidth]{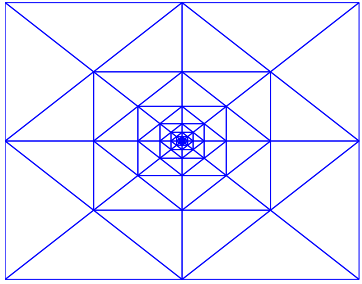}
  \end{center}
  \caption{$\mathcal T_{25}$}\label{mesh2}
\end{figure}
\begin{table}[H]
  \ttabbox
  {
    \begin{tabular}{p{10.pt} p{10.pt} p{10.pt} p{10.pt} p{10.pt} p{10.pt} p{10.pt} p{10.pt}}
      \toprule
      $k$ & $m_0$ & $\mathcal T_5$ & $\mathcal T_{10}$ & $\mathcal T_{15}$ & $\mathcal T_{20}$ & $\mathcal T_{25}$\\
      \midrule
      0 & 1 & 15 & 15 & 15 & 15 & 15\\
          & 2 & 12 & 12 & 12 & 12 & 12\\
          & 3 & 12 & 12 & 12 & 12 & 12\\
      \midrule
      1 & 1 & 30 & 30 & 30 & 30 & 30\\
          &2& 16 & 16 & 16 & 16 & 16\\
          &3& 12 & 12 & 12 & 12 & 12\\
      \bottomrule
    \end{tabular}
  }
  {
    \caption{Numerical results for the third experiment}\label{tab:third experiment}
  }
\end{table}

%\begin{rem}
   For the first two examples, the regularity estimate \eqref{eq:reg} holds with only $\alpha\leqslant 0.5$,
  which violates the regularity requirement $\alpha\in (0.5,1]$ in \cite{Cockburn;2014}.
  For the third example, not only \eqref{eq:reg} holds with $\alpha\leqslant 0.5$,
  but also the triangulation is not quasi-uniform. However, for all the experiments, the numerical results
  are consistent with our theoretical results, which shows that our algorithm is
  convergent even when $\alpha$ is not greater than $0.5$ in  \eqref{eq:reg}
  and the triangulation is
  not quasi-uniform.
%\end{rem}
\appendix
\section{Analysis of   symmetric Gauss-Seidel iteration}
Let $\mathcal R_h$ be the symmetric Gauss-Seidel iteration. As stated in Remark \ref{Rem3.3},  we can show $\mathcal R_h$ satisfies  {\bf Assumption \ref{assum:R_h 1}}. Suppose {\bf Assumption \ref{assum:a_h}} is true. 
Then by the well-known properties of Gauss-Seidel iteration,
we know that \eqref{eq:assum R_hA_h} holds with $\omega=1$. Thus it remains to verify \eqref{eq:assum inv R_h}.

Let $\{\eta_i:i=1,2,\cdots, N\}$ be the standard nodal basis for $M_h$. Define $P_i:M_h\to\text{span}\{\eta_i\}$ by
\begin{equation}
  \ninpro{P_i\lam_h,\eta_i} = \ninpro{\lam_h,\eta_i},\quad  i=1,2,\cdots, N.
\end{equation}
 By Theorem 3 in \cite{Chen;2010}, we have
\begin{equation}\label{eq:1 appendix}
  \inpro{\mathcal R_h^{-1}\mu_h,\mu_h}=\norm{\mu_h}^2_{\mathcal A_h}+\inf_{\sum_{i=1}^N\mu_i=\mu_h}
  \sum_{i=1}^N\norm{P_i\sum_{j>i}\mu_j}^2_{\mathcal A_h},~\forall\mu_h\in M_h.
\end{equation}
Then, by using the same technique used in the proof of Lemma 3.7, we can obtain
\begin{equation}\label{eq:2 appendix}
  \inpro{\mathcal R_h^{-1}\mu_h,\mu_h}\lesssim\sum_{T\in\mathcal T_h}h_T^{-2}\norm{\mu_h}^2_{h,\partial T}.
\end{equation}
Denote $\mathcal S_h:=\mathcal R_h\mathcal A_h$. By the definition  \eqref{def_sysm_R} of $\overline{\mathcal R_h}$,
it holds
\begin{equation}
  \overline{\mathcal R_h}=2\mathcal R_h-\mathcal R_h\mathcal A_h\mathcal R_h
  =(2\mathcal S_h-\mathcal S_h^2)\mathcal A_h^{-1},
\end{equation}
which yields
\begin{displaymath}
  \begin{split}
    \inpro{\overline{\mathcal R_h}^{-1}\lam_h,\lam_h}
    &=\inpro{\mathcal\mathcal R_h^{-1} S_h(2\mathcal S_h-\mathcal S_h^2)^{-1}\lam_h,\lam_h}.
  \end{split}
\end{displaymath}
It is easy to verify that $\mathcal S_h$ is symmetric with respect to the inner product $\inpro{\mathcal R_h^{-1}\cdot,\cdot}$.
Then, from the inequality 
\begin{equation}
  t(2t-t^2)^{-1}<1\text{ for all $t\in (0,1)$}
\end{equation}
and   the fact that all the eigenvalues of $\mathcal S_h$ are in $(0,1)$, it follows 
\begin{equation}\label{eq:3 appendix}
  \inpro{\overline{\mathcal R_h}^{-1}\lam_h,\lam_h}\leqslant\inpro{\mathcal R_h^{-1}\lam_h,\lam_h},
\end{equation}
which, together with  \eqref{eq:2 appendix}, leads to the desired result   \eqref{eq:assum inv R_h}.


\begin{thebibliography}{100}\small
\bibitem{Aksoylu;2006} B. AKSOYLU, M. HOLST, Optimality of multilevel preconditioners for
  local mesh refinement in three dimensions, SIAM J. Numer. Anal., \textbf{44} (2006), 1005-1025.
  \par
\bibitem{ArnoldBrezzi1985} D. N. ARNOLD, F. BREZZI, Mixed and non-conforming finite element
  methods: implementation, post-processing and error estimates, Mod\'el. Math. Anal. Num\'er.,
  19 (1985), 7-35.
  \par
\bibitem{Bornemann;1993} F. BORNEMANN, B. ERDMANN, R. KORNHUBER, Adaptive multilevel methods in
  three space dimensions, Int. J. for Numer. Meth. in Eng., \textbf{36} (1993), 3187-3203.
  \par
\bibitem{Brandt;1977}  A. BRANDT, Multi-level adaptive solutions to boundary-value problems,
  Math. Comp., \textbf{31} (1977), 333-390.
  \par
\bibitem{Brandt;1982} A. BRANDT, S. F. MCCORMICK, J. W. RUGE, Algebraic Multigrid (AMG) for
  Automatic Multigrid Solution with Application to Geodetic Computations, Technical report,
  Institute for Computational Studies, Colorado State University, Fort Collins, CO, 1982.
  \par
\bibitem{Brandt;1986} A. BRANDT, Algebraic multigrid theory: The symmetric case, Appl. Math.
  Comput., \textbf{19} (1986), 23-56.
  \par
\bibitem{BrezziDouglasMarini1985} F. BREZZI, J. DOUGLAS, JR., L. D. MARINI, Two families of mixed finite elements for second
  order elliptic problems, Numer. Math., 47 (1985), 217-235.

\bibitem{Chen;1994;a} Z. CHEN, Equivalence between and multigrid algorithms for mixed and
  nonconforming methods for second order elliptic problems, East-West J. Numer. Math.,
  \textbf{4} (1994), 1-33.
  \par
\bibitem{Chen;2010} L. CHEN, Deriving the X-Z identity from auxiliary space method, In: The
  Proceedings for 19th Conferences for Domain Decomposition Methods, 2010.
  \par
\bibitem{Chen;2011} L. CHEN, R. H. NOCHETTO, J. XU, Optimal multilevel methods for graded bisection
  grids, Numer. Math., \textbf{120} (2011), 1-34.
  \par
\bibitem{Chen;2014} L. CHEN, J. WANG, Y. WANG, X. YE, An auxiliary space multigrid preconditioner
  for the weak Galerkin method, arXiv preprint arXiv:1410.1012, 2014.
  \par
\bibitem{Xu;2008} D. CHO, J. XU, L. ZIKATANOV, New estimates for the rate of convergence of the
  method of subspace corrections, Numer. Math. Theor. Meth. Appl., \textbf{1} (2008), 44-56.
  \par
\bibitem{COCKBURN-GOPALAKRISHNAN2004}B. COCKBURN, J. GOPALAKRISHNAN, A characterization of hybridized mixed methods for
  second order elliptic problems, SIAM J. Numer. Anal., 42 (2004), 283-301.
  \par
\bibitem{COCKBURN_2005}B. COCKBURN, J. GOPALAKRISHNAN, New hybridization techniques, GAMM-Mitt,
  2 (2005), 154-183.
\par
\bibitem{Cockburn-Gopalakrishnan-Wang2007}B. COCKBURN, J. GOPALAKRISHNAN, AND H. WANG, Locally conservative fluxes
  for the continuous Galerkin method, SIAM J. Numer. Anal., 45 (2007), 1742-1776.
  \par
\bibitem{super_con_LDG} B. COCKBURN, B. DONG, J. GUZM\'AN, A superconvergent LDG-hybridizable
  Galerkin method for second-order elliptic problems, Math. Comp., 77 (2008), 1887-1916.
  \par
\bibitem{Cockburn;unified HDG;2009} B. COCKBURN, J. GOPALAKRISHNAN, R. LAZAROV, Unified hybridization
  of discontinuous Galerkin, mixed, and conforming Galerkin methods for second order elliptic problems,
  SIAM J. Numer. Anal., \textbf{47} (2009), 1319-1365.
  \par
\bibitem{projection_based_hdg} B. COCKBURN, J. GOPALAKRISHNAN, F. J. SAYAS, A projection-based error
  analysis of HDG methods, Math. Comp., 79 (2010), 1351-1367.
\par
\bibitem{Cockburn;2014} B. COCKBURN, O. DUBOIS, J. GOPALAKRISHNAN, S. TAN, Multigrid for an HDG
  method, IMA Journal of Numerical Analysis, \textbf{34} (2014), 1386-1425.
  \par
\bibitem{Dahmen;1992} W. DAHMEN, A. KUNOTH, Multilevel preconditioning, Numer. Math., \textbf{63}
  (1992), 315-344.
  \par
\bibitem{Gopalakrishnan;2003} J. GOPALAKRISHNAN, A Schwarz preconditioner for a hybridized mixed
  method, Comput. Methods Appl. Math., \textbf{3} (2003), 116-134.
  \par
\bibitem{Gopalakrishnan;2009} J. GOPALAKRISHNAN, S. TAN, A convergent multigrid cycle for the
  hybridized mixed method, Numer. Linear Algebra Appl., $\bm{16}$ (2009), 689-714.
  \par
\bibitem{Li;mg_wg;2014} B. LI, X. XIE, A two-level algorithm for the weak Galerkin discretization
  of diffusion problems, arXiv preprint arXiv:1405.7506v3, 2014.
  \par
\bibitem{Li;hdg;2014} B. LI, X. XIE, Analysis of a family of HDG methods for second order elliptic
  problems, arXiv preprint arXiv:1408.5545, 2014.
  \par
\bibitem{Li;bpx;2014} B. LI, X. XIE, BPX preconditioner for nonstandard finite element methods
  for diffusion problems, arXiv preprint parXiv:1410.5332v1.
  \par
\bibitem{Livne;2004}O. E. LIVNE, Coarsening by compatible relaxation, Numer. Linear Algebra Appl.,
  \textbf{11} (2004), 205-227.
  \par
\bibitem{McCormick;1984} S. MCCORMICK, Fast adaptive composite gird (FAC) methods: theory for the
  variational case, In Defect correction methods (Oberwolfach, 1983), volume 5 of Comput. Suppl.,
  Springer, Vienna, 1984, 115-121.
  \par
\bibitem{McCormick;1986} S. F. MCCORMICK, J. W. THOMAS, The fast adaptive composite grid (FAC)
  method for elliptic equations, Math. Comp., \textbf{46} (1986), 439-456.
  \par
\bibitem{Mu-Wang-Wang-Ye}L. MU, J. WANG, Y. WANG,  X. YE, A computational study of the weak Galerkin method
  for second-order elliptic equations, arXiv:1111.0618v1, 2011, Numerical Algorithms, 2012,
  DOI:10.1007/s11075-012-9651-1.
  \par
\bibitem{Mu-Wang-Ye1}L. MU, J. WANG, X. YE, A weak Galerkin finite element methods with polynomial
  reduction, arXiv:1304.6481, submited to SIAM J on Scientific Computing.
  \par
\bibitem{Mu-Wang-Ye2}L. MU, J. WANG,   X. YE, Weak Galerkin finite element methods on polytopal meshes,
  arXiv:1204.3655v2, submitted to International J of Nmerical Analysis and Modeling.
  \par
\bibitem{RT}P.-A. RAVIART, J. M. THOMAS, A mixed finite element method for second order elliptic
  problems, Mathematical Aspects of Finite Element Method (I. Galligani and E. Magenes, eds.),
  Lecture Notes in Math. 606, Springer-Verlag, New York, 1977, 292-315.
  \par
\bibitem{Vanek;1996} P. VAN\v{E}K, J. MANDEL, M. BREZINA, Algebraic multigrid by smoothed
  aggregation for second and fourth order elliptic problems, Computing, \textbf{56} (1996), 179-196.
  \par
\bibitem{Wang;2013} J. WANG, X. YE, A weak Galerkin finite element method for second-order
  elliptic problems, J. Comp. Appl. Math., \textbf{241} (2013), 103-115.
  \par
\bibitem{Wu;2006} H. WU, Z. CHEN, Uniform convergence of multigrid V-cycle on adaptively refined
  finite element meshes for second order elliptic problems, Science in China: Series A Mathematics,
  \textbf{49} (2006), 1-28.
  \par
\bibitem{Xu;2002} J. XU, L. ZIKATANOV, The method of alternating projections and the method of
  subspace corrections in Hilbert space, J. Am. Math. Soc., \textbf{15} (2002), 573-597.
  \par
\end{thebibliography}
\end{document}